\documentclass{amsart}

\usepackage[paperheight=11in, 
    paperwidth=8.5in,
    outer=1.35in,
    inner=1.35in,
    bottom=1.5in,
    top=1.5in]{geometry}

\usepackage{amsmath, amsxtra, amsthm, amsfonts, amssymb}
\usepackage{mathtools}
\usepackage{mathrsfs}
\usepackage{graphicx}
\usepackage{url}
\usepackage{color}
\usepackage{bbm}
\usepackage{tikz-cd}

\usepackage[T1]{fontenc}
\usepackage{enumitem}

 \usepackage[hidelinks,backref=page]{hyperref} 
 \hypersetup{
    colorlinks,
    citecolor=magenta,
    filecolor=magenta,
    linkcolor=blue,
    urlcolor=black
}

\usepackage{cleveref}

\usetikzlibrary{positioning}
\usetikzlibrary{matrix}
\usetikzlibrary{decorations}
\usetikzlibrary{decorations.pathreplacing, decorations.pathmorphing, angles,quotes}

\newcommand{\QQ}{\mathbb Q}
\newcommand{\RR}{\mathbb R}

\newcommand{\PP}{\mathbb P}
\newcommand{\ZZ}{\mathbb Z}
\newcommand{\kk}{\mathbbm{k}}
\newcommand{\M}{\mathrm{M}}
\renewcommand{\P}{\mathrm{P}}
\newcommand{\one}{\mathbf{1}}
\newcommand{\be}{\mathbf e}

\newcommand{\E}{E} 
\newcommand{\EE}{\mathsf{E}} 

\renewcommand{\emptyset}{\varnothing}
\newcommand{\rk}{\operatorname{rk}}
\newcommand{\expand}{\pi^*}
\newcommand{\lift}{\M_\pi}

\theoremstyle{definition}
\newtheorem{theorem}{Theorem}[section]
\newtheorem{definition}[theorem]{Definition}
\newtheorem{proposition}[theorem]{Proposition}
\newtheorem{lemma}[theorem]{Lemma}
\newtheorem{corollary}[theorem]{Corollary}

\newtheorem{remark}[theorem]{Remark}

\begin{document}

\title{Intersection theory of polymatroids}

\author{Christopher Eur}\address{Department of Mathematics, Harvard University, Cambridge, MA 02138, United States}\email{ceur@math.harvard.edu}
\author{Matt Larson}\address{Department of Mathematics, Stanford University, Stanford, CA 94305, United States}\email{mwlarson@stanford.edu}
 
\maketitle

\vspace{-20 pt}

\begin{abstract}
Polymatroids are combinatorial abstractions of subspace arrangements in the same way that matroids are combinatorial abstractions of hyperplane arrangements.  By introducing augmented Chow rings of polymatroids, modeled after augmented wonderful varieties of subspace arrangements, we generalize several algebro-geometric techniques developed in recent years to study matroids.  We show that intersection numbers in the augmented Chow ring of a polymatroid are determined by a matching property known as the Hall--Rado condition, which is new even in the case of matroids.
\end{abstract}

\section{Introduction}
\noindent
Let $\E= \{1, \dotsc, m\}$ be a finite set, and let $\textbf a = (a_1, \ldots, a_m)$ be a sequence of nonnegative integers.
\begin{definition}
A \emph{polymatroid} $\P$ on $\E$ with \emph{cage} $\textbf a$ is 
a function $\operatorname{rk}_\P \colon 2^\E \to \mathbb{Z}_{\ge 0}$ satisfying
\begin{enumerate}
\item (Submodularity) $\operatorname{rk}_\P(S_1) + \operatorname{rk}_\P(S_2) \ge \operatorname{rk}_\P(S_1 \cap S_2) + \operatorname{rk}_\P(S_1 \cup S_2)$ for any $S_1, S_2 \subseteq \E$,
\item (Monotonicity) $\operatorname{rk}_\P(S_1) \le \operatorname{rk}_\P(S_2)$ for any $S_1 \subseteq S_2\subseteq \E$,
\item (Normalization) $\operatorname{rk}_\P(\emptyset) = 0$, and
\item (Cage) $\operatorname{rk}_\P(i) \le a_i$ for any $i \in \E$. 
\end{enumerate}
We say that $\operatorname{rk}_\P$ is the \emph{rank function} of the polymatroid $\P$, and that $\P$ has \emph{rank} $r=\operatorname{rk}_\P(\E)$.
\end{definition}

A polymatroid with cage $(1, \dotsc, 1)$ is a \emph{matroid}.  For the fundamentals of matroid theory we point to \cite{Wel76}.
Introduced as generalizations of matroids \cite{Edm70}, and also known as generalized permutohedra, polymatroids are the central objects in the polyhedral study of combinatorial structures related to the symmetric group \cite{AA, Pos09}.
In those works, two polytopes associated to a polymatroid $\P = (\E,\operatorname{rk}_\P)$ are the \emph{independence polytope} $I(\P)$, defined by
\[
I(\P)  = \Big\{x \in \mathbb{R}^{\E}_{\ge 0} : \sum_{i \in S} x_i \le \operatorname{rk}_\P(S) \text{ for all }S \subseteq \E\Big\},
\]
and the \emph{base polytope} $B(\P)$, which is the face of $I(\P)$ defined by
\[
B(\P) = I(\P) \cap \Big\{x\in \RR^{\E} : \sum_{i\in \E} x_i = \operatorname{rk}_\P(\E)\Big\}.
\]
Both polytopes are re-encodings of the polymatroid $\P$ as follows \cite{Edm70}:
The polytope $B(\P)$ determines $I(\P)$ by $I(\P) = \{x \in \RR^\E_{\geq 0}: y - x \in \RR^\E_{\geq 0} \text{ for some }y\in B(\P)\}$, and the rank function of $\P$ is recovered by $\rk_\P(S) = \max\{ \sum_{i\in S}x_i : x\in I(\P)\} = \max\{ \sum_{i\in S}x_i : x\in B(\P)\}$.

\medskip
We connect polyhedral properties of polymatroids to algebro-geometric properties arising from the intersection theory of varieties associated to their realizations by linear subspaces.
Let $V_1, \ldots, V_m$ be vector spaces of dimensions $a_1, \ldots, a_m$ respectively over a field $\kk$, and let $V = \bigoplus_{i\in \E} V_i$.  A subspace $L\subseteq V$ defines a polymatroid $\P$ on $\E$ with cage $\textbf a = (a_1, \dotsc, a_m)$ whose rank function is
\[
\operatorname{rk}_\P(S) = \dim\Big(\text{image of $L$ under the projection $V\to \bigoplus_{i\in S}V_i$}\Big) \quad \text{for any $S\subseteq \E$}.
\]
We say that $L\subseteq V$ is a \emph{realization} of the polymatroid $\P$ in this case.
A realization $L\subseteq V$ defines a \emph{subspace arrangement} on $L$ that consists of subspaces $\{L_i\}_{i\in \E}$ where $L_i = \ker(L\to V_i)$.  In terms of the subspace arrangement, the rank function of $\P$ is equivalently described as
\[
\operatorname{rk}_\P(S) = \operatorname{codim}_L\Big(\bigcap_{i\in S} L_i\Big) \quad\text{for any $S\subseteq \E$}.
\]
The key geometric object for us is the following compactification of $L\subseteq V$.

\begin{definition}
The \emph{augmented wonderful variety} $W_L$ of a subspace $L \subseteq V = \bigoplus_{i\in \E} V_i$ is
\[
W_L = \text{the closure of the image of $L$ in $\prod_{\emptyset \subsetneq S \subseteq \E} \mathbb{P} \Big(\bigoplus_{i \in S} V_i\oplus \kk\Big)$},
\]
where the map $L \to \mathbb{P}(\bigoplus_{i \in S} V_i\oplus \kk)$ is the composition of the projection $L\to \bigoplus_{i\in S} V_i$ with the projective completion $\bigoplus_{i \in S} V_i \hookrightarrow \PP(\bigoplus_{i\in S} V_i \oplus \kk)$.  
\end{definition}

In the context of matroids and hyperplane arrangements, the augmented wonderful variety was introduced in \cite{BHMPW22}, and it played a central role in the proof of Dowling--Wilson top-heavy conjecture \cite{BHMPWb}.
Augmented wonderful varieties are closely related to the wonderful compactifications of subspace arrangement complements introduced in \cite{dCP95}. 

\smallskip
A special role is played by the \emph{boolean arrangement} $L = \bigoplus_{i \in \E} V_i$ with cage $\textbf{a}$, whose augmented wonderful variety is called the \emph{polystellahedral variety} with cage $\textbf{a}$, denoted $X_{\textbf{a}}$.
Let $A^\bullet(X_{\textbf a})$ be the Chow cohomology ring of $X_{\textbf a}$, which in \Cref{cor:polystellaChow} we show has the presentation
\[
A^{\bullet}(X_{\textbf{a}}) = \frac{\mathbb{Z}[x_S, y_i: \varnothing \subseteq S \subsetneq \E,\ i \in \E]}{\langle x_{S_1} x_{S_2} : S_1, S_2 \text{ incomparable} \rangle + \langle x_S y_i^{a_i} : i \not \in S \rangle + \langle y_i - \sum_{S \not \ni i} x_S : i \in \E \rangle}.
\]
Its grading satisfies $A^{\bullet}(X_{\textbf a}) = \bigoplus_{k=0}^{a_1 + \cdots + a_m} A^k(X_{\textbf a})$, and it is equipped with the degree map, which is an isomorphism
\[
\deg_{X_\textbf a}\colon A^{a_1 + \dotsb + a_m}(X_{\textbf{a}}) \overset\sim\to \mathbb{Z} \quad\text{determined by the property}\quad \deg_{X_{\textbf a}}(y_1^{a_1}\cdots y_m^{a_m}) = 1.
\]
The Chow homology group $A_\bullet(X_{\textbf a})$ is the graded group $\bigoplus_{k=0}^{a_1 + \dotsb + a_m} A_k(X_{\textbf a})$ where $A_k(X_{\textbf a}) = A^{a_1 + \dotsb + a_m - k}(X_{\textbf a})$.

\smallskip
For a polymatroid $\P = (\E,\rk_\P)$ with cage $\textbf{a}$ and rank $r$, we define a homology class $[\Sigma_P] \in A_r(X_{\textbf{a}})$ called the \emph{augmented Bergman class} of $\P$ (\Cref{defn:augBergclass}).
We define the \emph{augmented Chow ring} $A^\bullet(\P)$ of $\P$ by
\[
A^\bullet(\P) = A^{\bullet}(X_{\textbf{a}})/\operatorname{ann}([\Sigma_\P]), \text{ where }\operatorname{ann}([\Sigma_\P]) = \{x \in A^{\bullet}(X_{\textbf{a}}) : x \cdot [\Sigma_\P] = 0\}.
\]
See \Cref{cor:augChow} for an explicit presentation of $A^\bullet(\P)$.
Its grading satisfies $A^\bullet(\P) = \bigoplus_{k=0}^r A^k(\P)$, and it is equipped with the degree map, which is an isomorphism $\deg_{\P}\colon A^r(\P) \overset\sim\to \ZZ$ defined by
\[
\deg_\P(\xi) = \deg_{X_{\textbf a}}(\xi' \cdot [\Sigma_\P]) \text{ for any lift $\xi' \in A^{\bullet}(X_{\textbf{a}})$ of $\xi \in A^{\bullet}(\P)$}.
\]
When a subspace $L\subseteq V$ realizes $\P$, one has an embedding $W_L \hookrightarrow X_{\textbf{a}}$ by the construction of the augmented wonderful variety.
The resulting homology class $[W_L] \in A_r(X_{\textbf{a}})$ equals $[\Sigma_\P]$ (Proposition ~\ref{prop:realizable}), and the Chow ring $A^\bullet(W_L)$ of the augmented wonderful variety $W_L$ coincides with the augmented Chow ring $A^\bullet(\P)$ (Remark~\ref{rmk:chowequiv}).

\medskip
The embedding $X_{\textbf{a}} \hookrightarrow \prod_{\emptyset \subsetneq S \subseteq \E} \mathbb{P}(\bigoplus_{i \in S} V_i\oplus \kk)$ provides the following useful set of generators for the Chow ring of $X_{\textbf a}$.
For each nonempty subset $S \subseteq \E$, let $h_S \in A^1(X_{\textbf{a}})$ be the pullback of the hyperplane class on $\mathbb{P}(\bigoplus_{i \in S} V_i \oplus \kk)$ along the map induced by the embedding $X_{\textbf{a}} \hookrightarrow \prod_{\emptyset \subsetneq S \subseteq \E} \mathbb{P}(\bigoplus_{i \in S} V_i\oplus \kk)$.
We show that $\{h_S: \emptyset\subsetneq S\subseteq \E\}$ generates $A^\bullet(X_{\textbf a})$, and that the monomials in these generators are all of the form $[\Sigma_\P]$ for some polymatroid $\P$ with cage $\textbf a$.
For a polymatroid $\P$, we define $h_S \in A^1(\P)$ to be image of $h_S$ under the quotient map $A^{\bullet}(X_{\textbf{a}}) \to A^{\bullet}(\P)$.
We call these the \emph{simplicial generators} of $A^\bullet(\P)$, motivated by similar terminology in the case of matroids \cite{BES, LLPP}.  These generators were also considered in \cite{Yuz02}.

\smallskip
We show that the intersection numbers of the simplicial generators are described by the \emph{Hall--Rado} condition:  A sequence $S_1, \dotsc, S_r$ of nonempty subsets of $\E$ is said to satisfy the {Hall--Rado} condition (with respect to a polymatroid $\P = (\E,\rk_\P)$) if 
\[
\rk_\P \Big( \bigcup_{j \in J} S_j \Big) \ge |J| \quad\text{for all}\quad J \subseteq \{1, \ldots, r\}.
\]
See \Cref{lem:HR} for an interpretation of this condition in terms of a matching problem.

\begin{theorem}\label{thm:HR}
Let $\P$ be a polymatroid of rank $r$, and let $S_1, \dotsc, S_r$ be a sequence of nonempty subsets of $\E$. Then
$$\deg_\P(h_{S_1} \dotsb h_{S_r}) = \begin{cases} 1 & S_1, \dotsc, S_r \text{ satisfies the Hall--Rado condition,} \\ 
0 & \text{otherwise.} \end{cases}$$
\end{theorem}

At least when $\P$ is realizable, the fact that $\deg_\P(h_{S_1} \dotsb h_{S_r}) = 0$ if $S_1, \dotsc, S_r$ does not satisfy the Hall--Rado condition has a simple geometric explanation. If $\operatorname{rk}_{\P}(S_{i_1} \cup \dotsb \cup S_{i_k}) < k$, then the degree $k$ element $h_{S_{i_1}} \dotsb h_{S_{i_k}}$ is zero because it is pulled back from the image of $W_L$ in $\mathbb{P}(\bigoplus_{i \in S_{i_1}} V_i \oplus \kk) \times \dotsb \times \mathbb{P}(\bigoplus_{i \in S_{i_k}} V_i \oplus \kk)$, which has dimension $\operatorname{rk}_{\P}(S_{i_1} \cup \dotsb \cup S_{i_k}) < k$.

\medskip
We highlight here the following corollary of \Cref{thm:HR}.

\begin{corollary}\label{cor:basisgenfct}
Let $\P$ be a polymatroid on $\E$ of rank $r$.  Then $\frac{1}{r!} \deg_{\P} \big( (\sum_{i\in \E} t_i h_{\{i\}})^r\big)$, the \emph{volume polynomial} of $A^\bullet(\P)$ with respect to $\{h_{\{i\}}: i\in \E\} \subset A^1(\P)$, equals the \emph{basis exponential generating function} of $\P$, which is the polynomial in $\mathbb{Q}[t_i: i\in \E]$ given by
\[
\sum_{\mathbf u\in B(\P)\cap \ZZ^\E} \frac{t^{\mathbf u}}{\mathbf u!}, \quad\text{where $t^{\mathbf u} = t_1^{u_1}\cdots t_m^{u_m}$ and ${\mathbf u}! = u_1!\cdots u_m!$}.
\]
\end{corollary}

Our results here generalize several previous results in the literature.
\begin{itemize}
\item When $\P$ is realizable and has cage $(1, \ldots, 1)$, \Cref{cor:basisgenfct} specializes to \cite[Theorem 1.3(c)]{Ardila-Boocher}.
\item When $\P$ is realizable, \Cref{thm:HR} specializes to \cite[Proposition 7.15]{CCMM} and the first statement of \cite[Theorem 1.1]{LiImages}.
\item When $\P$ has cage $(1, \ldots, 1)$, \Cref{thm:dHR} (a variant of \Cref{thm:HR}) specializes to \cite[Theorem 5.2.4]{BES}.  When $\P$ is also boolean, it further specializes to \cite[Theorem 9.3]{Pos09} because intersection numbers on toric varieties can be interpreted as mixed volumes. 
\end{itemize}

\smallskip
Many invariants of matroids behave well with respect to matroid polytope decompositions. This leads to the study of the \emph{valuative group} of matroids \cite{AFR10, BEST, DF10}, which gives a powerful tool to study invariants of matroids. We consider the following notion of valuativity for polymatroids with cage $\textbf a$.

\begin{definition}
For a polytope $Q\subset \RR^\E$, let $\one_Q\colon \RR^\E \to \ZZ$ be its indicator function defined by $\one_Q(x) = 1$ if $x\in Q$ and $\one_Q(x) = 0$ otherwise.
The \emph{valuative group} $\operatorname{Val}_r(\mathbf{a})$ of rank $r$ polymatroids with cage $\textbf{a}$ is the subgroup of $\mathbb{Z}^{(\mathbb{R}^\E)}$ generated by $\one_{B(\P)}$ for $\P$ a polymatroid of rank $r$ and with cage $\mathbf{a}$. 
\end{definition}

We show that the valuative group is isomorphic to the homology groups of the polystellahedral variety, generalizing \cite[Theorem 1.5]{EHL}.

\begin{theorem}\label{thm:val}
For any $0\leq r \leq a_1 + \dotsb + a_m$, the map that sends a polymatroid $\P$ with cage $\textbf{a}$ and rank $r$ to $[\Sigma_{\P}]$ induces an isomorphism $\operatorname{Val}_{r}(\textbf a) \overset\sim\to A_r(X_{\textbf{a}})$. 
\end{theorem}

To prove \Cref{thm:val}, we show that a choice of isomorphism $V_i \simeq \kk^{a_i}$ for each $i\in \E$ realizes $X_{\textbf{a}}$ as a toric variety (\Cref{prop:toricisom}).
This gives a description of the Grothendieck ring of vector bundles $K(X_{\textbf{a}})$ in terms of certain polytopes in $\mathbb{R}^{a_1 + \dotsb + a_m}$ (\Cref{ssec:polytopealg}). We relate $\bigoplus_r \operatorname{Val}_r(\textbf{a})$ to this polytopal description. We then prove an exceptional Hirzebruch--Riemann--Roch-type theorem (\Cref{thm:polystellaHRR}) that leads to the proofs of both Theorems~\ref{thm:HR} and \ref{thm:val}.

\medskip
The paper is organized as follows. In Section 2, we discuss polystellahedral varieties from the point of view of toric geometry. In Section 3, we construct the augmented Bergman fan of a polymatroid and develop its basic properties. In Section 4, we study the $K$-ring of the polystellahedral variety. In Section 5, we prove Theorem~\ref{thm:HR} and~\ref{thm:val}. In Section 6, we prove analogs of Theorem~\ref{thm:HR} and \ref{thm:val} for the polypermutohedral variety.

\subsection*{Acknowledgements}
We thank June Huh for many invaluable conversations related to polymatroids, including suggesting the statements of Theorems~\ref{thm:HR} and \ref{thm:val}.
We thank the referees for many helpful comments. 
The first author is supported by NSF Grant DMS-2001854, and the second author is supported by an NDSEG fellowship. 

\subsection*{Notations} 
All varieties are over an algebraically closed field $\kk$. 
For a subset $S\subseteq \{1, \dotsc, \ell\}$, let $\be_S = \sum_{i\in S} \be_i$ be the sum of standard basis vectors in $\RR^\ell$.  Denote by $(\cdot, \cdot)$ the standard inner product.  For polyhedra and toric varieties, we follow conventions of \cite{Ful93, CLS11}. For a rational polyhedral fan $\Sigma$, we let $X_{\Sigma}$ be the toric variety associated to $\Sigma$.

\section{The toric geometry of polystellahedral varieties}

We introduce the \emph{polystellahedral fan} (with cage $\textbf a$) and study the properties of the associated toric variety.
This amounts to developing basic properties of the polystellahedral variety $X_{\textbf a}$, since we will show that any choice of isomorphisms $V_i \simeq \kk^{a_i}$ for all $i\in \E$ induces an isomorphism between $X_{\textbf a}$ and the toric variety associated to the polystellahedral fan.

\subsection{Polystellahedral fans}

Set $n = a_1 + \dotsb + a_m$, and let $\EE$ be a set of cardinality $n$.  A map $\pi\colon \EE \to E$, which defines a partition $\EE = \bigsqcup_{i\in \E} \pi^{-1}(i)$, is said to have cage $\textbf a$ if $|\pi^{-1}(i)| = a_i$ for all $i\in \E$.

\begin{definition}\label{def:polystellfan}
A \emph{compatible pair} with respect to a map $\pi\colon \EE \to \E$ is a pair $I\leq \mathcal F$ consisting of a subset $I\subseteq \EE$ and a chain $\mathcal F = \{F_1 \subsetneq F_2 \subsetneq \cdots \subsetneq F_k \subsetneq F_{k+1} = \E\}$ of proper subsets of $\E$ such that if $\pi^{-1}(S) \subseteq I$ for a subset $S\subseteq \E$, then $S\subseteq F_1$.

The \emph{polystellahedral fan} $\Sigma_{\pi}$ is the fan in $\mathbb{R}^{\EE}$ whose cones are in bijection with compatible pairs, with a compatible pair $I \le \mathcal{F}$ corresponding to the cone
$$\sigma_{I \le \mathcal{F}} = \operatorname{cone}(-\be_{\EE \setminus \pi^{-1}(F_1)}, \dotsc, -\be_{\EE \setminus \pi^{-1}(F_k)}) + \operatorname{cone}(\be_{i}: i\in I).$$
Its rays are denoted $\rho_i = \mathbb{R}_{\ge 0} \be_i$ for $i \in \EE$ and $\rho_{F} = \mathbb{R}_{\ge 0}( -\be_{\EE \setminus \pi^{-1}(F)})$ for $\emptyset\subseteq F \subsetneq \E$.
\end{definition}

Note that the fan $\Sigma_\pi$ depends only on the map $\EE \to \pi(\EE)$, not the codomain $\E$ of $\pi$.
A polystellahedral fan $\Sigma_{\textbf a}$ with cage $\textbf a$ is a fan $\Sigma_\pi$ where $\pi$ has cage $\textbf a$.
We note two extreme cases:
\begin{itemize}
\item When $\pi$ has cage $(n)$, the fan $\Sigma_\pi$ is the inner normal fan of the $n$-dimensional standard simplex $\operatorname{conv}(\{0\} \cup \{\be_j: j\in \EE\})$ in $\RR^{\EE}$.  We denote this fan by $\Sigma_n$.
\item When $\pi$ has cage $(1, \ldots, 1)$, the fan $\Sigma_\pi$ is the \emph{stellahedral fan} on $\EE$ in \cite{EHL}.  We denote this fan by $\Sigma_{\EE}$.
\end{itemize}
A general polystellahedral fan in $\RR^\EE$ is both a refinement of $\Sigma_n$ and a coarsening of $\Sigma_\EE$ in the following way.
For two maps $\pi\colon \EE \to E$ and $\pi'\colon \EE \to \E'$, let us say $\pi$ \emph{refines} $\pi'$, denoted $\pi\succeq \pi'$, if the corresponding partitions form a refinement, i.e., for every $i\in \E$ one has $\pi^{-1}(i) \subseteq {\pi'}^{-1}(i')$ for some $i'\in \E'$.
Recall that for a simplicial fan $\Sigma$ and a vector $v$ in its support, the \emph{stellar subdivision} of $\Sigma$ by $v$ is the new fan whose set of rays are $\{\text{rays of }\Sigma\} \cup \{\rho_v = \RR_{\geq 0}v\}$ and the set of cones are
$
\{\sigma\in \Sigma: v\notin \sigma\} \cup \{\sigma \cup \rho_v : \sigma\in \Sigma \text{ such that } v\notin \sigma \text{ and } v\in \sigma' \text{ for some } \sigma\subset \sigma'\in \Sigma\}
$.

\begin{proposition}\label{prop:polystellafan}
For a refinement $\pi \succeq \pi'$, let $(S_1, \ldots, S_k)$ be a sequence consisting of the subsets $S\subseteq \E$ such that ${\pi}^{-1}(S) \neq {\pi'}^{-1}(S')$ for any $S'\subseteq \E'$, ordered in a way that $|S_1| \geq \cdots \geq |S_k|$.  Then the fan $\Sigma_{\pi}$ is the result of the sequence of stellar subdivisions of the fan $\Sigma_{\pi'}$ by the sequence of vectors $(-\be_{\EE\setminus {\pi}^{-1}(S_1)}, \ldots, -\be_{\EE\setminus{\pi}^{-1}(S_k)})$.  Moreover, at each step of the sequence of stellar subdivisions, the resulting fan is projective and unimodular with respect to the lattice $\ZZ^{\EE}$.
\end{proposition}

We prove the proposition using \emph{building sets}, which were introduced in \cite{dCP95} and studied in \cite{FY04, FS05}.
We first review the special case of building sets on a boolean lattice here following \cite[Section 7]{Pos09}, which is simpler than the general case. We will discuss building sets in a more general context in Section~\ref{ssec:augbergman}.
A \emph{building set} on $\E$ is a collection $\mathcal G \subseteq 2^\E$ of subsets of $\E$ such that $\mathcal G$ contains $\E$ and $\{i\}$ for each $i\in \E$, and if $S$ and $S'$ are in ${\mathcal G}$ and $S\cap S' \neq \emptyset$, then $S\cup S' \in \mathcal{G}$.
A \emph{nested set} of ${\mathcal G}$ is a collection $\{X_1, \ldots, X_k\} \subseteq {\mathcal G}$ such that for every subcollection $\{X_{i_1}, \ldots, X_{i_\ell}\}$ with $\ell \geq 2$ consisting only of pairwise incomparable elements, one has $\bigcup_{j=1}^\ell X_{i_j} \notin {\mathcal G}$.
The fan associated to $\mathcal G$ is the fan $\Sigma_{\mathcal G}$ in $\RR^\E/ \RR \be_\E$ whose cones are
\[
\{\text{the image in $\RR^\E/\RR\be_\E$ of } \operatorname{cone}\{\be_{X_1}, \ldots, \be_{X_k}\} \subset \RR^E: \{X_1, \ldots, X_k\} \text{ a nested set of $\mathcal G$}\}.
\]

\begin{proof}
Let $\EE\cup \{0\}$ be the disjoint union of $\EE$ with an extra element 0. We have an isomorphism $\RR^{\EE \cup \{0\}}/\RR \be_{\EE \cup \{0\}} \simeq \RR^{\EE}$ induced by $\be_i \mapsto \be_i$ for $i\in \EE$ and $\be_0 \mapsto -\sum_{i\in \EE} \be_i$.  It is straightforward to verify that, under this isomorphism, the fan $\Sigma_\pi$ equals the fan $\Sigma_{\mathcal G_\pi}$ in $\RR^{\EE \cup 0}/\RR \mathbf 1$ associated to the building set $\mathcal G_\pi = \{\{i\} : i \in\EE\} \cup \{\pi^{-1}(S) \cup 0 : \emptyset\subseteq S\subseteq \E\}$ on the boolean lattice of $\EE \cup \{0\}$.  If $\pi\succeq \pi'$, then we have $\mathcal G_\pi \supseteq \mathcal G_{\pi'}$, and the desired statements in the proposition are now special cases of \cite[Theorem 4.2]{FM05} and \cite[Proposition 2]{FY04}.
\end{proof}

\subsection{Polystellahedral varieties}
Let us fix the following notation for the rest of a paper.

\smallskip
\noindent\textbf{Notation.} Let $\EE$ be a set of size $n \coloneqq a_1 + \dotsb + a_m$, and let $\pi\colon \EE \to E$ be a map with cage $\textbf a$.

\smallskip
Let $X_\pi$ be the toric variety associated to the polystellahedral fan $\Sigma_\pi$.
We record some properties of $X_\pi$ arising from the properties of the fan $\Sigma_\pi$, starting with the fact that $X_\pi$ is isomorphic to the polystellahedral variety $X_{\textbf a}$ with cage $\textbf a$.

\medskip
As before, let $V = \bigoplus_{i\in \E}V_i$ be the direct sum of vector spaces where $\dim V_i = a_i = |\pi^{-1}(i)|$ for all $i\in \E$.  Denote by $GL_{\textbf a}$ the group $ \prod_{i\in \E} GL(V_i)$.
Recall  that $X_{\mathbf a}$ is the closure of the image of the map $V \to \prod_{\emptyset\subsetneq S \subseteq \E} \PP(\bigoplus_{i\in S} V_i \oplus \kk)$.  Because this map is $GL_{\textbf a}$-equivariant, the group $GL_{\textbf a}$ acts naturally on the variety $X_{\textbf a}$.

\begin{proposition}\label{prop:toricisom}
Any choice of isomorphisms $V_i \simeq \kk^{\pi^{-1}(i)}$ for each $i\in \E$, which gives a natural embedding of the torus $(\kk^*)^\EE \hookrightarrow GL_{\textbf a}$, identifies $X_{\textbf a}$ with the toric variety $X_\pi$ of the fan $\Sigma_\pi$.
\end{proposition}

Thus, from this point on, we will identify $X_{\textbf a}$ with the toric variety $X_\pi$, although the identification depends on the choices of isomorphisms $V_i \simeq \kk^{\pi^{-1}(i)}$ for all $i\in \E$.

\begin{proof}
With the isomorphisms $V_i \simeq \kk^{\pi^{-1}(i)}$ for all $i\in \E$, the projective space $\PP^{\EE} = \PP(\kk^\EE \oplus \kk) \simeq \PP(V \oplus \kk)$ with the obvious action of $(\kk^*)^\EE$ is the toric variety of the fan $\Sigma_n$.
For a subset $S\subseteq \E$, let $L_S = \kk^{\pi^{-1}(E\setminus S)} \oplus 0 \subset \kk^{\pi^{-1}(E\setminus S)} \oplus \kk$.  If $S$ is a proper subset, then $\PP(L_{S})$ is the hyperplane at infinity of the coordinate subspace $ \PP(\kk^{\pi^{-1}(E\setminus S)} \oplus \kk) \simeq \PP(\bigoplus_{i\in E\setminus S} V_i \oplus \kk)$ of $\PP^\EE$.
Note the complementation, and note that $\PP(L_S)$ is $GL_{\textbf a}$-invariant for any $\emptyset\subseteq S\subsetneq E$.

We apply \Cref{prop:polystellafan} with $\pi'$ being the map from $\EE$ to a singleton set, which describes the fan $\Sigma_\pi$ as a sequence of stellar subdivisions of the fan $\Sigma_n$.  Translated into toric geometry terms, it states that the toric variety $X_\pi$ of the fan $\Sigma_\pi$ is obtained from $\PP^\EE$ via a sequence of blow-ups as follows:  Order the proper subsets of $\E$ so that their cardinalities are non-strictly decreasing, then sequentially blow-up the (strict transforms of) the loci $\PP(L_{S})$ in that order.  This sequential blow-up is also the description of the \emph{wonderful compactification} of the complement of the subspace arrangement $\{\PP(L_{\{i\}}): i\in \E\}$ in $\PP^\EE$, introduced in \cite{dCP95}.  \cite[\S1.6 Proposition (2)]{dCP95} moreover states that this wonderful compactification is also the closure of the image of the rational map $\PP^\EE \dashrightarrow \prod_{\emptyset\subsetneq S \subseteq E} \PP\big((\kk^\EE\oplus \kk)/L_S\big)$, which, when restricted to $V \simeq \kk^\EE \subset \PP^\EE$, is exactly the map $V \to \prod_{\emptyset\subsetneq S \subseteq \E} \PP(\bigoplus_{i\in S} V_i \oplus \kk)$.
\end{proof}

\begin{remark}\label{rem:gamma1}
Let $\Gamma_{\textbf a}$ be the product $\prod_{i\in \E} \mathfrak S_{\pi^{-1}(i)}$ of permutation groups.
Because $\Gamma_{\textbf a}$ acts naturally on the fan $\Sigma_\pi$ by permuting the coordinates of $\RR^\EE$, the group $\Gamma_{\textbf a}$ acts on the variety $X_\pi$.
Under the identification $X_{\textbf a} \simeq X_\pi$, this action agrees with the action of $\Gamma_{\textbf a}$ embedded in $GL_{\textbf a}$ via the isomorphism $\bigoplus_{i\in \E} V_i \simeq \bigoplus_{i\in \E} \kk^{\pi^{-1}(i)}$.
\end{remark}

We record the following presentation of the Chow ring of $X_{\textbf a}$.
For a proper subset $S$ of $\E$ and an element $j\in \EE$, 
let $x_S$ and $\widetilde y_j$ denote the toric divisors of $X_{\textbf a}$ corresponding to the rays $\rho_S$ and $\rho_j$ of $\Sigma_{\textbf a}$, respectively.

\begin{corollary}\label{cor:polystellaChow}
For each $i\in \E$, the divisors in the set $\{\widetilde y_j: j\in \pi^{-1}(i)\}$ are all equal to each other as divisor classes in $A^1(X_{\textbf a})$.  Denote this divisor class by $y_i$.
The Chow ring $A^\bullet(X_{\textbf a})$ of $X_\pi$ equals
\[
A^{\bullet}(X_{\pi}) = \frac{\mathbb{Z}[x_S, y_i: \varnothing \subseteq S \subsetneq \E,\ i \in \E]}{\langle x_{S_1} x_{S_2} : S_1, S_2 \text{ incomparable} \rangle + \langle x_S y_i^{a_i} : i \not \in S \rangle + \langle y_i - \sum_{S \not \ni i} x_S : i \in \E \rangle}.
\]
\end{corollary}

\begin{proof}
For a unimodular and projective fan $\Sigma$ in $\RR^\EE$ with rays $\Sigma(1)$ and primitive ray vectors $\{u_\rho \in \ZZ^\EE : \rho \in \Sigma(1)\}$, \cite[\S5.2 Proposition]{Ful93} states that the Chow ring of the smooth projective toric variety $X_\Sigma$ equals
\[
A^\bullet(X_\Sigma) =\frac{\ZZ[x_\rho : \rho \in \Sigma(1)]}{\langle \prod_{\rho\in S} x_\rho : \{\rho_i\}_{i\in S} \text{ do not form a cone in $\Sigma$}\rangle + \langle \sum_{\rho\in \Sigma(1)} (u_\rho, v) x_\rho: v\in \ZZ^\E \rangle}
\]
where $(u,v)$ here denotes the standard inner product on $\RR^\EE$ and $x_\rho$ represents the toric divisor of $X_\Sigma$ corresponding to the ray $\rho$.
We apply this with $\Sigma = \Sigma_\pi$.

Setting $v = \be_{j_1} - \be_{j_2}$ for any $i\in \E$ and $j_1, j_2 \in \pi^{-1}(i)$, the linear relations $\sum_{\rho\in \Sigma(1)} (u_\rho, v) x_\rho = 0$ imply the first statement that $\{\widetilde y_j\}_{j\in \pi^{-1}(i)}$ are all equal as elements in $A^1(X_\pi)$.
Setting $v = \be_j$ for any $i\in \E$ and $j\in \pi^{-1}(i)$ then gives the relations $\{y_i - \sum_{S \not \ni i} x_S  = 0 : i \in \E \}$.
The rest of the corollary follows when one notes that the minimal non-faces of $\Sigma_{\pi}$ are the following:
the sets of the form $\{\rho_{S_1}, \rho_{S_2}\}$ for incomparable proper subsets $S_1$ and $S_2$ of $\E$, or the sets of the form $\{\rho_S\} \cup \{ \rho_j : j \in \pi^{-1}(i)\}$ for a proper subset $S$  of $\E$ and $i \in E\setminus S$.
\end{proof}

\subsection{Nef divisors, deformations, and expansions}\label{ssec:nef}
For a fan $\Sigma$ in $\RR^\EE$, a (lattice) polytope $Q\subset \RR^\EE$ is a \emph{(lattice) deformation} of $\Sigma$ if its inner normal fan $\Sigma_Q$ coarsens the fan $\Sigma$.
We describe the deformations of the polystellahedral fan.

\medskip
As before, let $\pi\colon \EE \to \E$ be a map with cage $\textbf a$.
Define a linear map
\[
p_\pi\colon \RR^\EE \to \RR^\E \quad\text{by}\quad \be_i \mapsto \be_{\pi(i)}.
\]

\begin{definition}
Let $\P = (\E, \rk_\P)$ be a polymatroid on $\E$ with arbitrary cage.  The \textbf{expansion} (with respect to $\pi$) of $\P$ is the polymatroid $\expand(\P)$ on $\EE$ whose rank function is given by $\rk_\P \circ \pi$.  Equivalently, the polymatroid $\expand(\P)$ is defined by setting its independence polytope to be
\[
I(\expand(\P)) = p_\pi^{-1}(I(\P)) \cap \RR^\EE_{\geq 0}.
\]
\end{definition}

\begin{proposition}\label{prop:deformations}
A lattice polytope $Q$ in $\RR^\EE$ is a deformation of $\Sigma_{\textbf{a}}$ if and only if $Q$ is a translate of  $I(\expand(\P))$ for a polymatroid $\P$ on $\E$. 
\end{proposition}

We deduce the proposition by using a standard result in toric geometry that identifies deformations with nef toric divisors.
We prepare with the following lemma.
Note that, by the linear relations for the Chow ring $A^\bullet(X_{\textbf a})$ in \Cref{cor:polystellaChow}, the set of divisor classes $\{x_S: \emptyset \subseteq S\subsetneq \E\}$ is a basis of $A^1(X_{\textbf{a}})$.

\begin{lemma}\label{lem:nef}
A divisor class $D \in A^1(X_{\textbf{a}})$ is nef if and only if, when we write $D = \sum_{S \subsetneq E} a_{S} x_S$, the function $S \mapsto a_{E \setminus S}$ is the rank function of a polymatroid on $\E$.
\end{lemma}

\begin{proof}
Let $\varphi_D$ be the piecewise linear function corresponding to the divisor $D = \sum_{S \subsetneq E} a_S x_S$, which satisfies $\varphi_D(\be_j) = 0$ for all $j \in \EE$ and $\varphi_D(-\be_{\EE \setminus \pi^{-1}(S)}) = -a_S$ for $S \subsetneq E$. We use a criterion for the nefness of a line bundle on a smooth projective toric variety from \cite[Theorem 6.4.9]{CLS11}, which states that $D$ is nef if and only if the support function $\varphi_D$ satisfies an inequality for each minimal non-face of the fan. This gives the following inequalities:
\begin{itemize}
   \item For $S, S' \subsetneq E$ incomparable, the minimal non-face spanned by $\rho_{S}$ and $\rho_{S'}$ gives the inequality
   $$\varphi_D(-\be_{\EE \setminus \pi^{-1}(S)} - \be_{\EE \setminus \pi^{-1}(S')}) \ge \varphi_D(-\be_{\EE \setminus \pi^{-1}(S)}) + \varphi_D(-\be_{\EE \setminus \pi^{-1}(S')}).$$
   Because $-\be_{\EE \setminus \pi^{-1}(S)} - \be_{\EE \setminus \pi^{-1}(S')} = -\be_{\EE \setminus \pi^{-1}(S \cap S')} - \be_{\EE \setminus \pi^{-1}(S \cup S')}$ and $\varphi_D$ is linear on the cone spanned by $-\be_{\EE \setminus \pi^{-1}(S \cap S')}$ and $- \be_{\EE \setminus \pi^{-1}(S \cup S')}$, we get that
   $$a_{S \cap S'} + a_{S \cup S'} \le a_{S} + a_{S'}.$$
   \item For $S \subsetneq E$ and $i \not \in S$, the minimal non-face spanned by $\rho_S \cup \{ \rho_j : j \in \pi^{-1}(i)\}$ gives the inequality
   $$\varphi_D(-\be_{\EE \setminus \pi^{-1}(S)} + \sum_{j \in \pi^{-1}(i)} \be_j) \ge \varphi_D(-\be_{\EE \setminus \pi^{-1}(S)}) + \sum_{j \in \pi^{-1}(i)} \varphi_D(\be_j).$$
   As $-\be_{\EE \setminus \pi^{-1}(S)} + \sum_{j \in \pi^{-1}(i)} \be_j = -\be_{\EE \setminus \pi^{-1}(S \cup i)}$ and $\varphi_D(\be_j) = 0$, this gives the inequality
   $$a_{S \cup i} \le a_{S}.$$
\end{itemize}
These two inequalities are equivalent to the statement that $S \mapsto a_{E \setminus S}$ is a polymatroid. 
\end{proof}

\begin{proof}[Proof of \Cref{prop:deformations}]
The standard correspondence between nef toric divisors and deformations \cite[Theorems 6.1.5--6.1.7]{CLS11}, when applied to the fan $\Sigma_{\textbf a}$, states that
a nef divisor $D = \sum_{S \subsetneq E} a_S x_S$ on $X_{\textbf a}$ corresponds to the lattice deformation $Q_D$ of $\Sigma_{\textbf a}$ defined by
\[
Q_D = \{y\in \RR^\EE: (y, \be_j ) \ge 0 \text{ for all $j\in \EE$ and }( y, -\be_{\EE \setminus \pi^{-1}(S)} ) \ge -a_S \text{ for all $\emptyset\subseteq S \subsetneq \E$}\},
\]
which is exactly the independence polytope of the expansion of the polymatroid with rank function $S \mapsto a_{\E \setminus S}$.
Moreover, the correspondence implies that every lattice deformation of $\Sigma_{\textbf a}$ arises as a translate of the polytope corresponding to a nef divisor $D = \sum_{S \subsetneq E} a_S x_S$.
\end{proof}

We distinguish the following set of nef divisors on $X_{\textbf a}$ arising from the standard simplices in $\RR^\E$.  Note that, for each nonempty subset $S\subseteq \E$, the simplex $\Delta_S^0 = \operatorname{conv}(\{0\} \cup \{\be_i: i\in S\}) \subset \RR^\E$ is the independence polytope of the polymatroid on $\E$ whose rank function is
\[
\rk(T) = \begin{cases}
1 &\text{if $T\cap S \neq\emptyset$}\\
0 &\text{otherwise}
\end{cases} \qquad\text{for $\emptyset\subseteq T \subseteq \E$,}
\]
or equivalently, $\rk(\E\setminus T) = 1$ exactly when $T\not\supseteq S$.

\begin{definition}
For each nonempty subset $S\subseteq \E$, we define $h_S \in A^1(X_{\textbf a})$ to be the nef divisor
\[
h_S = \sum_{\substack{\emptyset\subseteq T\subsetneq \E \\ T\not\supseteq S}}x_T 
\]
corresponding to the simplex $\Delta_S^0$.
We call the divisor classes $\{h_S: \emptyset\subsetneq S \subseteq \E\}$ the \emph{simplicial generators} of $X_{\textbf a}$.
\end{definition}

\begin{proposition}\label{prop:simpgen}
The simplicial generators of $X_{\textbf a}$ form a basis of $A^1(X_{\textbf a})$. In particular, their monomials span $A^\bullet(X_{\textbf a})$ as an abelian group.
\end{proposition}

\begin{proof}
By M\"obius inversion, every divisor class in the basis $\{x_T: \emptyset\subseteq T \subsetneq \E\}$ of $A^1(X_{\textbf a})$ is a linear combination of the simplicial generators.
\end{proof}

\begin{remark}
The definition of $h_S$ here agrees with its definition in the introduction as the pullback of the hyperplane class of $\PP(\bigoplus_{i\in S} V_i \oplus \kk)$ along the map
$$ \textstyle X_{\textbf{a}} \hookrightarrow \prod_{\emptyset \subsetneq S \subseteq \E} \mathbb{P}(\bigoplus_{i \in S} V_i\oplus \kk) \to \mathbb{P}( \bigoplus_{i \in S} V_i \oplus \kk).$$
To see this, one notes that the independence polytope of the expansion of the polymatroid of $\Delta_S^0$ is the simplex $\Delta_{\pi^{-1}(S)}^0 = \operatorname{conv}(\{0\} \cup \{\be_j : j\in \pi^{-1}(S)\}) \subset \RR^\EE$.
The lattice points of $\Delta_{\pi^{-1}(S)}^0$, considered as global sections of the corresponding line bundle, induce the map $X_{\textbf a} \to \PP(\bigoplus_{i\in S} V_i \oplus \kk)$.
\end{remark}

We conclude by discussing the behavior of Chow rings under refinements.
\Cref{prop:polystellafan} implies that $\Sigma_\pi$ is a coarsening of the stellahedral fan $\Sigma_\EE$.  Thus, we have a toric birational map
\[
u\colon X_\EE \to X_{\textbf a} \quad\text{induced by the refinement of fans}\quad \Sigma_\EE \succeq \Sigma_\pi.
\]
We record the following properties of $u$ for future use.

\begin{lemma}\label{lem:pullback}
The pullback map $u^*\colon A^\bullet(X_{\textbf a}) \to A^\bullet(X_\EE)$ satisfies the following.
\begin{enumerate}
\item $u^*$ is a split injection, with the splitting given by the pushforward map $u_*\colon A^\bullet(X_\EE) \to A^\bullet(X_{\textbf a})$.
\item If $D\in A^1(X_{\textbf a})$ is a nef divisor class corresponding to a deformation $Q$ of $\Sigma_{\textbf a}$, then the pullback $u^*D \in A^1(X_\EE)$ is a nef divisor class corresponding to $Q$ considered as a deformation of $\Sigma_\EE$.
\item For a nonempty subset $S\subseteq \E$, the simplicial generator $h_S \in A^1(X_{\textbf a})$ pulls back to the simplicial generator $u^* h_S = h_{\pi^{-1}(S)} \in A^1(X_\EE)$.
\end{enumerate}
\end{lemma}

\begin{proof}
The first statement is a standard consequence of the birationality of $u$ and the projection formula.  The second statement follows from \cite[Proposition 6.2.7]{CLS11}.  The third statement follows from the second, since the independence polytope of the expansion of the polymatroid of $\Delta_S^0$ is the simplex $\Delta_{\pi^{-1}(S)}^0 = \operatorname{conv}(\{0\} \cup \{\be_j : j\in \pi^{-1}(S)\}) \subset \RR^\EE$.
\end{proof}

\begin{remark}
Let the \emph{polystellahedron} with cage $\textbf a$ be the polytope $\Pi_{\textbf a}$ in $\RR^\EE$ defined by
\begin{align*}
\Pi_{\textbf a} = I(\expand(\P)), \text{ where $\P$ is the polymatroid on $\E$ with } \text{$B(\P) = \operatorname{conv}\{w\cdot(1,\ldots,m) : w\in \mathfrak S_\E\}$}.
\end{align*}
The face $B(\expand(\P))$ of $\Pi_{\textbf a}$ was introduced as the \emph{polypermutohedron} with cage $\textbf a$ in \cite{CHLSW}.
Using the results in this subsection, one can verify that the polystellahedral fan $\Sigma_{\textbf a}$ is the normal fan of the polystellahedron $\Pi_{\textbf a}$.  Alternatively, using the building set associated to a polystellahedral fan given in the proof of \Cref{prop:polystellafan}, one can verify that $\Pi_{\textbf a}$ is the corresponding nestohedron \cite[Section 7]{Pos09}.
\end{remark}

\section{Augmented geometry of polymatroids}

For a polymatroid $\P$ with cage $\textbf a$ and rank $r$, we define its {augmented Bergman fan} $\Sigma_\P$ as a subfan of the polystellahedral fan with cage $\textbf a$, and we use its properties to define the augmented Bergman class $[\Sigma_\P] \in A_r(X_{\textbf a})$.
We then record some geometric properties of the augmented Bergman fan and the augmented Bergman class.

\subsection{Multisymmetric lifts and duality}

We begin with a construction of a matroid from a polymatroid $\P$ with cage $\textbf{a}$ which has appeared many times in the literature \cite{H72,L77,M75,N86,BCF} under different names, such as the ``free expansion'' and ``natural matroid.'' Here, we use the terminology of \cite{CHLSW}.

\begin{definition}
The \emph{multisymmetric lift} of a polymatroid $\P$ on $\E$ with cage \textbf{a} is the matroid $\lift(\P)$ on $\EE$ whose rank function is given by
\[
\rk_{\lift(\P)}(S) = \min\{ \rk_\P(A) + |S\setminus \pi^{-1}(A)| : A\subseteq \E\}.
\]
\end{definition}

Alternatively, the multisymmetric lift can be described via polytopes as follows.

\begin{lemma}\label{lem:projection}
Let $[0,1]^\EE$ be the unit cube in $\RR^\EE$.  Then, we have
$I(\lift(\P)) = I(\expand(\P)) \cap [0,1]^\EE$.
\end{lemma}

\begin{proof}
We need to show that a subset $S \subseteq \EE$ is independent in the matroid $\lift(\P)$ if and only if $\be_S \in I(\expand(\P))$.
By the definition of $I(\expand(\P))$, we have that $\be_S \in I(\expand(\P))$ if and only if, for all $U \subseteq \EE$, one has $|S \cap U| \le \operatorname{rk}_\P(\pi(U))$.
It suffices to check whether this holds when $U$ is a fiber of $\pi$, so this condition becomes $|S \cap \pi^{-1}(A)| \le \operatorname{rk}_\P(A)$, or, equivalently, $\rk_\P(A) + |S\setminus \pi^{-1}(A)| = \rk_\P(A) + |S| - |S\cap \pi^{-1}(A)| \geq |S|$, for all $A \subseteq E$.
That is, the condition is equivalent to $\rk_{\lift(\P)}(S) = |S|$.
\end{proof}

The lemma and its proof implies that $I(\lift(\P))$ maps onto $I(\P)$ under the linear projection $p_\pi\colon \RR^\EE \to \RR^\E$.
Equivalently, a polymatroid $\P$ with cage $\textbf a$ is recovered from $\lift(\P)$ via the formula $\operatorname{rk}_\P(S) = \operatorname{rk}_{\lift(\P)}(\pi^{-1}(S))$.

\medskip
When $\P$ is realized by a subspace arrangement $L \subseteq \bigoplus_{i\in \E}V_i$, the multisymmetric lift $\lift(\P)$ is realized by the hyperplane arrangement $L \subseteq \kk^\EE$ obtained by a general choice of isomorphisms $V_i \simeq \kk^{\pi^{-1}(i)}$ for all $i\in \E$.
In particular, the subspaces $\{L_i: i\in \E\}$ in the arrangement appear as subspaces arising as intersections of the hyperplanes in the arrangement $L\subseteq \kk^\EE$.

\smallskip
For a polymatroid $\P = (\E, \rk_\P)$ with cage \textbf{a}, a subset $F \subseteq \E$ is a \emph{flat} of $\P$ if $\operatorname{rk}_\P(F \cup a) > \operatorname{rk}_\P(F)$ for all $a \in  \E\setminus F$.
The flats of $\P$ form a lattice, denoted $\mathcal{L}_\P$. The \emph{loops} of a polymatroid are the elements of the minimal flat. 
We say that a polymatroid is \emph{loopless} if the empty set is a flat, or equivalently, if $\operatorname{rk}_\P(i) > 0$ for all $i \in \E$. 
Given a flat $F$ of $\P$, the subset $\pi^{-1}(F)\subseteq \EE$ is a flat of the multisymmetric lift $\lift(\P)$.
Flats of $\lift(\P)$ of this form are called \emph{geometric flats} of $\lift(\P)$.
The key property of geometric flats is the following. 

\begin{proposition}\cite[Lemma 2.8]{CHLSW}\label{lem:georank}
Every flat $F$ of $\lift(\P)$ contains a unique maximal geometric flat $F^{\mathrm{geo}}$. We have that $\operatorname{rk}_{\lift(\P)}(F^{\mathrm{geo}}) = \operatorname{rk}_{\operatorname{P}}(\pi(F^{\mathrm{geo}}))$, and 
$\operatorname{rk}_{\lift(\P)}(F) = \operatorname{rk}_{\lift(\P)}(F^{\mathrm{geo}}) + |F \setminus F^{\mathrm{geo}}|.$
\end{proposition}

\begin{remark}
As in \Cref{rem:gamma1}, let $\Gamma_{\textbf a}$ be the product $\prod_{i\in E} \mathfrak S_{\pi^{-1}(i)}$ of permutation groups.  The terminology ``multisymmetric'' is justified by the fact that 
the obvious action of the group $\Gamma_{\textbf a}$ on $\EE$ preserves the rank function of $\lift(\P)$. In fact, this property characterizes multisymmetric lifts: \cite[Theorem 2.9]{CHLSW} states that a matroid $\lift$ on $\EE$ such that the action of $\Gamma_{\textbf a}$ preserves the rank function is of the form $\lift(\P)$ for a polymatroid $\P$ with cage $\textbf{a}$.\footnote{In the proof of this theorem, the authors of \cite{CHLSW} make the additional assumption that $\operatorname{rk}_\P(i) = a_i$, but this assumption is never used.}
Moreover, the map $F \mapsto \pi^{-1}(F)$ induces an isomorphism from the lattice $\mathcal{L}_{\operatorname{P}}$ of flats of $\P$ to the lattice of $\Gamma_{\textbf a}$-fixed flats of $\lift(\P)$ \cite[Corollary 2.7]{CHLSW}. 
\end{remark}
\noindent
We now discuss polymatroid duality, see, e.g., \cite{M75}. Our main conclusion is that taking multisymmetric lift commutes with polymatroid duality.

\begin{definition}
For a polymatroid $\P$ on $\E$ with cage $\textbf a$ and rank $r$, its \emph{dual polymatroid} $\P^\perp$ is a polymatroid on $\E$ with cage $\textbf a$ and rank $n-r$ whose rank function is
\[
\operatorname{rk}_{\P^\perp}(S) = \sum_{i\in S} a_i + \operatorname{rk}(E\setminus S) - r.
\]
\end{definition}

Alternatively, duality can also be described via polytopes as follows.  
The rank function description for $\P^\perp$ above implies that
\[
B(\P^\perp) = -B(\P) + \textbf a,
\]
or, equivalently, since $I(\P) = \{ x\in \prod_{i\in \E} [0,a_i]: y-x \in \RR_{\geq 0}^\E \text{ for some $y\in B(\P)$}\}$, we have
\[
-I(\P^\perp) + \textbf a = \{ x\in  \textstyle \prod_{i\in \E} [0,a_i] : x - y \in \RR_{\geq 0}^\E\text{ for some $y\in B(\P)$}\}.
\]

\begin{figure}[h]
\centering
\includegraphics[height=42mm]{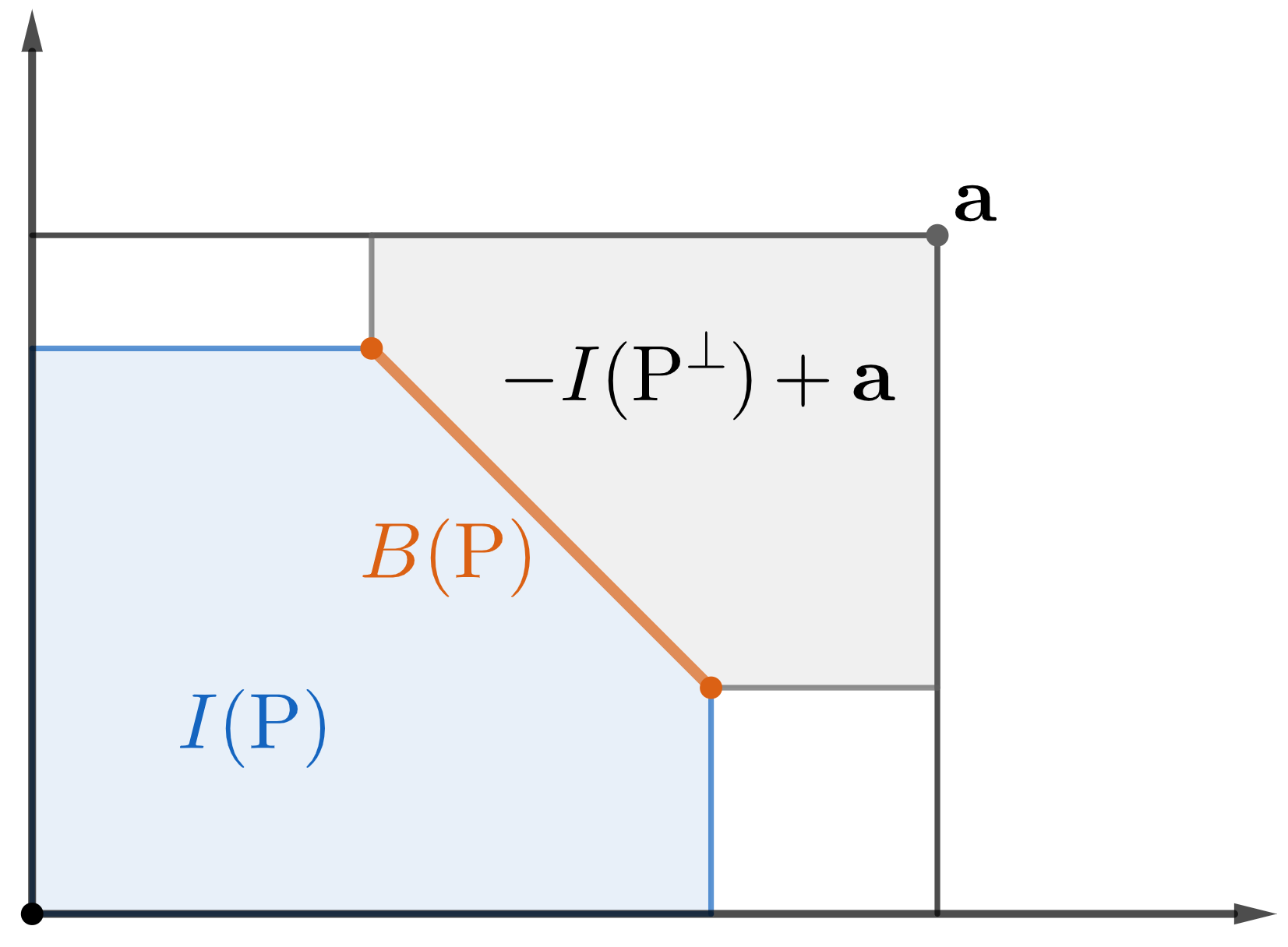}
\caption{Polytopes associated to a polymatroid and its dual.}
\end{figure}

When $\P$ is realized by $L\subseteq V = \bigoplus_{i\in \E} V_i$, its dual $\P^\perp$ is realized by $(V/L)^\vee \subseteq \bigoplus_{i\in \E} V_i^\vee$ obtained by dualizing the surjection $V \twoheadrightarrow V/L$.
When $\textbf a = (1, \ldots, 1)$, polymatroid duality agrees with the usual notion of matroid duality.

\begin{proposition}\label{prop:dual}
For a polymatroid $\P$ on $\E$ with cage $\textbf a$, one has $\lift(\P^\perp) = \lift(\P)^\perp$.
\end{proposition}

\begin{proof}
This follows from \Cref{lem:projection} since $B(\P^\perp) = -B(\P) + \textbf a$ and $p_\pi(\sum_{j\in \EE}\be_j) = \textbf a$.
\end{proof}

\subsection{Augmented Bergman fans of polymatroids}\label{ssec:augbergman}

Let $\P$ be a polymatroid on $\E$ with cage $\textbf{a}$. We now introduce the {augmented Bergman fan} $\Sigma_\P$ of a polymatroid.

\begin{definition}
The \emph{augmented Bergman fan} $\Sigma_\P$ of $\P$ is the subfan of $\Sigma_{\textbf{a}}$ consisting of cones $\sigma_{S \le \mathcal{F}}$, where $S$ is a subset of $\EE$ and $\mathcal{F} = \{F_1 \subsetneq \dotsb \subsetneq F_k \subsetneq F_{k+1} = E\}$ is a chain of proper flats of $\P$ satisfying
\begin{enumerate}
\item For all $T \subseteq S$, one has $\operatorname{rk}_{\P}(\pi(T)) \ge |T|$, and
\item for all $F \in \mathcal{F}$ and all nonempty $T \subseteq S \setminus \pi^{-1}(F)$, one has $\operatorname{rk}_{\P}(F \cup \pi(T)) > \operatorname{rk}_{\P}(F) + |T|$. 
\end{enumerate}
\end{definition}

When $\textbf a = (1, \dotsc, 1)$, that is, when $\P$ is a matroid $\M$ on $\E$, the augmented Bergman fan of $\P$ coincides with the augmented Bergman fan $\Sigma_\M$ introduced in \cite{BHMPW22}.  Explicitly, the fan $\Sigma_\M$ is the subfan of the stellahedral fan $\Sigma_\E$ consisting of cones $\sigma_{I\leq \mathcal F}$ where $I\subseteq \E$ is an independent set of $\M$ and $\mathcal F = \{F_1 \subsetneq \cdots \subsetneq F_k \subsetneq F_{k+1}= \E\}$ is a chain of proper flats of $\M$ such that $I\subseteq F_1$.

\begin{theorem}\label{thm:augbergmanfan}
The augmented Bergman fan $\Sigma_\P$ of $\P$ is the subfan of $\Sigma_{\textbf a}$ whose support is equal to the support of the augmented Bergman fan $\Sigma_{\lift(\P)}$ of the multisymmetric lift of $\P$.  More precisely, $\Sigma_\P$ is the coarsening of the fan $\Sigma_{\lift(\P)}$ such that it is a subfan of $\Sigma_{\textbf a}$.
\end{theorem}

This is the key property of $\Sigma_\P$ that we will repeatedly use.
The rest of this subsection is dedicated to the proof of the theorem.

\medskip
We now review building sets on the lattice of flats of a matroid; for proofs and details we point to \cite{dCP95, FS05}.
A \emph{building set} on a loopless matroid $\M$ on ground set $\E$ is a collection $\mathcal{G}$ of nonempty flats of $\M$ such that, for all nonempty flats of $F$ of $\M$, the natural map of lattices
$$\prod_{G \in \max \mathcal{G}_{\le F}} [\emptyset, G] \to [\emptyset, F]$$
is an isomorphism.
Here, $\max\mathcal G_{\leq F}$ denotes the maximal elements of $\mathcal G$ contained in the interval $[\emptyset, F] \subseteq \mathcal L_\M$.
All building sets that we consider will contain the maximal flat $E$.
A \emph{nested set} is a subset $\mathcal{N} \subseteq \mathcal{G}$ that does not contain $E$ such that, for all pairwise incomparable subsets $\{F_1, \dotsc, F_k\} \subseteq \mathcal{N}$ with $k \ge 2$, the join $\bigvee_{i=1}^k F_i$ of $\{F_1, \dotsc, F_k\}$ is not in $\mathcal{G}$.
Nested sets form a simplicial complex, which is realized as a simplicial fan $\underline{\Sigma}_{\M, \mathcal{G}}$ in $\RR^\E/\RR\be_\E$ whose cones are $\{\text{image in $\RR^\E/\RR \be_\E$ of }\operatorname{cone}\{\be_i: i\in \mathcal N\} \subset \RR^\E: \mathcal N \text{ a nested set}\}$.
We call $\underline{\Sigma}_{\M,\mathcal G}$ the \emph{Bergman fan} of $\M$ with respect to the building set $\mathcal G$.
The support of $\underline{\Sigma}_{\M, \mathcal{G}}$ does not depend on the choice of building set \cite[Theorem 4]{FY04}, and $\underline{\Sigma}_{\M, \mathcal{G}}$ is always a unimodular fan \cite[Proposition 2]{FY04}. 

\medskip
We prove \Cref{thm:augbergmanfan} by identifying the fan $\Sigma_\P$ with a Bergman fan of a matroid closely related to the multisymmetric lift $\lift(\P)$.
Let $\lift(\P) \times 0$ denote the \emph{free coextension} of the multisymmetric lift $\lift(\P)$, which is a matroid on the ground set $\EE \sqcup \{0\}$ with flats
$$\{F \cup 0  : F \subseteq \EE \text{ flat of }\lift(\P)\} \cup \{I \subseteq \EE : I \text{ independent in }\lift(\P)\}.$$
Note that $\lift(\P) \times 0$ is always loopless. We now define a building set on $\lift(\P) \times 0$ whose Bergman fan will be the augmented Bergman fan of $\P$. 

\begin{lemma}\label{lem:isbuilding}
Let $\mathcal{G}$ be the set of all flats of $\lift(\P) \times 0$ of the form $F \cup 0$ for $F$ a geometric flat of $\lift(\P)$, or $\{j\}$ for $j \in \EE$ not a loop of $\lift(\P)$. Then $\mathcal{G}$ is a building set.
\end{lemma}

\begin{proof}
Consider a flat of $\lift(\P) \times 0$ of the form $H \cup 0$ for $H$ a flat of $\lift(\P)$. By Lemma~\ref{lem:georank}, $H$ contains a unique maximal geometric flat $H^{\mathrm{geo}}$, and, for any subset $S$ with $H^{\operatorname{geo}} \subseteq S \subseteq H$, we have that $\operatorname{rk}_{\operatorname{P}}(S) = \operatorname{rk}_{\operatorname{P}}(H^{\operatorname{geo}}) + |S \setminus H^{\operatorname{geo}}|$. This identifies the interval $[\emptyset, H \cup 0]$ in $\mathcal{L}_{\lift(\P) \times 0}$ with $[\emptyset, H^{\operatorname{geo}} \cup 0] \times [\emptyset, H \setminus H^{\operatorname{geo}}]$. The second factor splits as $[\emptyset, H \setminus H^{\operatorname{geo}}] = \prod_{i \in H \setminus H^{\operatorname{geo}}} [\emptyset, i]$, which gives the desired decomposition for $H \cup 0$. 
If we have a flat of $\lift(\P) \times 0$ of the form $I$ for $I \subseteq \EE$ independent, then the desired decomposition is automatic. 
\end{proof}

Before computing the nested sets of $\mathcal{G}$, we need a preparatory lemma. 

\begin{lemma}\label{lem:findgeometric}
Let $F$ be a geometric flat of a multisymmetric matroid $\lift(\P)$, and let $S$ be a nonempty subset of $F$ such that $|S| \ge \operatorname{rk}_{\lift(\P)}(F)$ or $|S| > \operatorname{rk}_{\lift(\P)}(F)$. Then there is a geometric flat $G$ of $\lift(\P)$ and a nonempty subset $S' \subseteq S \cap G$ such that $|S'| \ge \operatorname{rk}_{\lift(\P)}(G)$ (respectively $|S'| > \operatorname{rk}_{\lift(\P)}(G)$) and $S'$ spans $G$. 
\end{lemma}

\begin{proof}
We first do the case when $|S| > \operatorname{rk}_{\lift(\P)}(F)$. We induct on the rank of $F$; if $\operatorname{rk}_{\lift(\P)}(F) = 0$ then the claim is obvious. Let $H$ be the closure of $S$. Using Lemma~\ref{lem:georank}, we have that
$$\operatorname{rk}_{\lift(\P)}(H) = \operatorname{rk}_{\lift(\P)}(H^{\mathrm{geo}}) + |H \setminus H^{\mathrm{geo}}| \ge \operatorname{rk}_{\lift(\P)}(H^{\mathrm{geo}}) + |S| - |S \cap H^{\mathrm{geo}}|.$$
On the other hand, we have that $\operatorname{rk}_{\lift(\P)}(H) \le \operatorname{rk}_{\lift(\P)}(F) < |S|$, so $\operatorname{rk}_{\lift(\P)}(H^{\mathrm{geo}}) < |S \cap H^{\mathrm{geo}}|$. Either $H^{\mathrm{geo}} = F$ and we are done, or we conclude by induction.  

In the case when $|S| \ge \operatorname{rk}_{\lift(\P)}(F)$, if we set $H$ to be the closure of $S$ then the argument above shows that $|S \cap H^{\operatorname{geo}}| \ge \operatorname{rk}_{\lift(\P)}(H^{\mathrm{geo}})$, so we are done unless $\operatorname{rk}_{\lift(\P)}(H^{\mathrm{geo}}) = 0$ (when $S \cap H^{\operatorname{geo}}$ may be empty). In this case, we have that $\operatorname{rk}_{\lift(\P)}(H) \ge |S| \ge \operatorname{rk}_{\lift(\P)}(F)$ by Lemma~\ref{lem:georank}, so $H = F$ is geometric. 
\end{proof}

\begin{lemma}\label{lem:nestedisaug}
With $\mathcal{G}$ as in Lemma~\ref{lem:isbuilding}, the nested sets of $\mathcal{G}$ are given by chains of flats $\mathcal{F} = \{F_1 \subsetneq \dotsb \subsetneq F_k \subsetneq F_{k+1}= \E\}$ of $\P$ and a subset $S$ of the non-loops of $\lift(\P)$ such that:
\begin{enumerate}
\item For all $T \subseteq S$, $\operatorname{rk}_{{\P}}(\pi(T)) \ge |T|$, and
\item for all $F \in \mathcal{F}$ and all nonempty $T \subseteq S \setminus \pi^{-1}(F)$, $\operatorname{rk}_{\operatorname{P}}(F \cup \pi(T)) > \operatorname{rk}_{\operatorname{P}}(F) + |T|$. 
\end{enumerate}
\end{lemma}

\begin{proof}
Let $S$ and $\mathcal{F} = \{F_1 \subsetneq \dotsb \subsetneq F_k \subsetneq F_{k+1} =  E\}$ be a pair satisfying the two condition of the lemma. We check that the corresponding set is nested. The incomparable subsets are either given by a collection $T \subseteq S$, or a flat $F \in \mathcal{F}$ and $T \subseteq S \setminus \pi^{-1}(F)$. 

The closure of $T \subseteq S$ in $\lift(\P) \times 0$ is $T$ if $T$ is independent, and it is $\operatorname{cl}_{\lift(\P)}(T) \cup 0$ if $T$ is dependent. In the first case, $T$ is not in $\mathcal{G}$ if $|T| > 1$. If $T$ is dependent, then (1) guarantees that $\operatorname{rk}_{\lift(\P)}(T) < \operatorname{rk}_\P(\pi(T))$, so the closure is not in $\mathcal{G}$. Similarly, if we have $T \subseteq S \setminus \pi^{-1}(F)$, then the closure of $\pi^{-1}(F) \cup T$ cannot be geometric. 

Now let $\mathcal{N}$ be a nested set, which consists of a subset $S$ of the non-loops of $\lift(\P)$ and flats of the form $\pi^{-1}(F) \cup 0$ for $F$ a flat of $\P$. As the join of two geometric flats is a geometric flat, the flats of $\P$ such that $\pi^{-1}(F) \cup 0$ lies in $\mathcal{N}$ must form a chain $\mathcal{F}$. 

Suppose there is a nonempty subset $T \subseteq S$ with $\operatorname{rk}_{\operatorname{P}}(\pi(T)) < |T|$. Let $F = \operatorname{cl}_\P(\pi(T))$ be the closure of $\pi(T)$, which is a flat of $\P$ of rank less than $|T|$ with $\pi^{-1}(F)$ containing $T$. By Lemma~\ref{lem:findgeometric}, there is $T' \subseteq T$ and a geometric flat $G$ such that $T'$ spans $G$ and $|T'| > \operatorname{rk}_{\lift(\P)}(G)$. Then the closure of $T'$ in $\lift(\P) \times 0$ is $G \cup 0$, contradicting that $\mathcal{N}$ is nested. 

Now suppose that there is $F \in \mathcal{F}$ and $T \subseteq S \setminus \pi^{-1}(F)$ with $\operatorname{rk}_{\operatorname{P}}(F \cup \pi(T)) \le \operatorname{rk}_{\operatorname{P}}(F) + |T|$. Let $G = \pi^{-1}(\operatorname{cl}_{\operatorname{P}}(\pi(T) \cup F))$. Applying Lemma~\ref{lem:findgeometric} to the contraction $\lift(\P)/\pi^{-1}(F)$, we find a geometric flat $H \supset \pi^{-1}(F)$ and $T' \subseteq T \cap H$ such that $T' \cup \pi^{-1}(F)$ spans $H$. This contradicts that $\mathcal{N}$ is nested. 
\end{proof}

\begin{proof}[Proof of \Cref{thm:augbergmanfan}]
Let $\mathcal{G}$ be the building set on the lattice of flats of $\lift(\P) \times 0$ given by Lemma~\ref{lem:isbuilding}.
Let $\mathcal{H}$ be the building set on the lattice of flats of $\lift(\P) \times 0$ given by $F \cup 0$ for $F$ a flat of $\lift(\P)$ and $\{j\}$ for $j \in \EE$ not a loop of $\lift(\P)$. That this is a building set follows from Lemma~\ref{lem:isbuilding} by viewing $\lift(\P)$ as a polymatroid with cage $(1, \dotsc, 1)$. 
By \cite[Theorem 4]{FY04} the support of $\underline{\Sigma}_{\lift(\P) \times 0, \mathcal{G}}$ coincides with the support of $\underline{\Sigma}_{\lift(\P) \times 0, \mathcal{H}}$. By \cite[Lemma 5.14]{EHL}, under the isomorphism $\mathbb{R}^{\EE} \to \mathbb{R}^{\EE \cup 0}/\mathbb{R}$ obtained by sending $\be_j$ to $\be_j$, the support of $\underline{\Sigma}_{\lift(\P) \times 0, \mathcal{H}}$ coincides with the support of $\Sigma_{\lift(\P)}$. Under this isomorphism, $\Sigma_{\lift(\P) \times 0, \mathcal{G}}$  is identified with $\Sigma_{\P}$ by Lemma~\ref{lem:nestedisaug}. 
\end{proof}

\subsection{Augmented Bergman classes of polymatroids}

We begin by reviewing briefly balanced fans and their Chow homology classes; for details and proofs we point to \cite{FS97} and \cite[Section 5]{AHK18}.

\medskip
A pure-dimensional simplicial rational fan $\Sigma$ of dimension $d$ is \emph{balanced} if for any cone $\tau\in \Sigma$ of codimension 1, one has $\sum_{\sigma \supsetneq \tau} u_{\sigma\setminus \tau} \in \tau$, where $u_{\sigma\setminus \tau}$ denotes the primitive vector of the unique ray in $\sigma$ that is not in $\tau$.  Suppose a balanced fan $\Sigma$ is a subfan of a complete unimodular fan ${\widetilde\Sigma}$.  Let $A^d(X_{{\widetilde\Sigma}})$ be the $d$-th graded piece of the Chow ring of the toric variety $X_{{\widetilde\Sigma}}$, which is spanned by $\{[Z_\sigma]: \sigma \text{ a $d$-dimensional cone in $\widetilde\Sigma$}\}$, where $Z_\sigma$ is the torus-orbit closure in $X_{{\widetilde\Sigma}}$ corresponding to $\sigma$.  One then obtains a linear functional $w_\Sigma \in \operatorname{Hom}(A^d(X_{\widetilde\Sigma}),\ZZ)$ determined by $w_\Sigma( [Z_\sigma]) = 1$ if $\sigma\in \Sigma$ and $w_\Sigma([Z_\sigma]) = 0$ otherwise.
By the Poincar\'e duality property of the Chow ring $A^\bullet(X_{\widetilde\Sigma})$, the functional $w_\Sigma$ defines an element $[\Sigma] \in A_d(X_\Sigma)$.

\medskip
Returning to polymatroids, let $\P$ be a polymatroid on $\E$ with cage $\textbf a$ and rank $r$.
As the support of the augmented Bergman fan $\Sigma_{\P}$ coincides with the support of a Bergman fan, \cite[Theorem 3.8]{Gross} implies that $\Sigma_\P$ is a balanced subfan of the polystellahedral fan with cage $\textbf a$.

\begin{definition}\label{defn:augBergclass}
The \emph{augmented Bergman class} of $\P$ is the Chow homology class $[\Sigma_\P] \in A_r(X_{\textbf a})$ obtained by considering $\Sigma_\P$ as a balanced subfan of the polystellahedral fan with cage $\textbf a$.
\end{definition}

We will repeatedly use the following relation between the classes associated to a polymatroid and its multisymmetric lift.
Recall the birational map $u\colon X_\EE \to X_{\textbf a}$ induced by refinement of respective fans (\Cref{prop:polystellafan}).

\begin{lemma}\label{lem:augpullback}
The pullback $u^* [\Sigma_\P]$ is equal to the augmented Bergman class $[\Sigma_{\lift(\P)}]$ of the multisymmetric lift. 
\end{lemma}

\begin{proof}
The lemma follows from applying the formula \cite[Corollary 3.7]{FS97} for computing pullbacks in terms of Minkowski weights to \Cref{prop:polystellafan} and \Cref{thm:augbergmanfan}.
\end{proof}

We use the lemma to compute how augmented Bergman classes of polymatroids multiply as elements in the Chow ring $A^\bullet(X_{\textbf a})$.
We will need the following combinatorial notions.

\medskip
Given two polymatroids $\P_1$ and $\P_2$ on $\E$ with cage $\textbf{a}$, we define the \emph{polymatroid union} $\P_1 \vee \P_2$ to be the polymatroid with cage $\textbf a$ whose independence polytope is $(I(\P_1) + I(\P_2)) \cap \prod_{i\in \E} [0, a_i]$.
That this is indeed the independence polytope of a polymatroid follows from \cite[(35)]{Edm70}.
Define the \emph{polymatroid intersection} of $\P_1$ and $\P_2$ to be $\P_1 \wedge \P_2 \coloneqq (\P_1^{\perp} \vee \P_2^{\perp})^{\perp}$.
If we view $\lift(\P_i)$ as a polymatroid with cage $(1, \dotsc, 1)$, by \Cref{lem:projection} we have that $\lift(\P_1) \vee \lift(\P_2) = \lift(\P_1 \vee \P_2)$. Therefore $\lift(\P_1) \wedge \lift(\P_2) = \lift(\P_1 \wedge \P_2)$ by \Cref{prop:dual}.

\begin{theorem}\label{thm:polymatroidintersect}
Let $\P_1$ and $\P_2$ be polymatroids with cage $\textbf{a}$ and ranks $r_1$ and $r_2$, respectively.  Then, we have
$$[\Sigma_{\P_1}] \cdot [\Sigma_{\P_2}] = \begin{cases} [\Sigma_{\P_1 \wedge \P_2}] & \text{ if } (n - r_1) + (n - r_2) = n - \operatorname{rank}(\P_1 \wedge \P_2) \\ 
0 & \text{ otherwise.}\end{cases}$$
\end{theorem}

When $\textbf{a} = (1, \ldots, 1)$, the above theorem is \cite[Theorem 1.6]{EHL}.  Our proof is a reduction to this case.

\begin{proof}
Applying \Cref{lem:augpullback} and using that $\lift(\P_1) \wedge \lift(\P_2) = \lift(\P_1 \wedge \P_2)$,
one obtains from \cite[Theorem 1.6]{EHL} that
$$u^*[\Sigma_{\P_1}] \cdot u^*[\Sigma_{\P_2}] = \begin{cases} u^*[\Sigma_{\P_1 \wedge \P_2}] & \text{ if } (n - r_1) + (n - r_2) = n - \operatorname{rank}(\P_1 \wedge \P_2) \\ 
0 & \text{ otherwise.}\end{cases}$$
The result now follows from the injectivity of $u^*$ (Lemma~\ref{lem:pullback}). 
\end{proof}

\begin{corollary}\label{cor:augspan}
The augmented Bergman classes of polymatroids with cage $\textbf{a}$ span $A^{\bullet}(X_{\textbf{a}})$ as an abelian group. 
\end{corollary}

\begin{proof}
Recall that $A^{\bullet}(X_{\textbf{a}})$ is generated as a ring by the simplicial generators $\{h_S\}$, and in particular, the monomials in the $\{h_S\}$ span $A^\bullet(X_{\textbf a})$ as an abelian group.
By \Cref{thm:polymatroidintersect}, we are done once we show that each simplicial generator $h_S$ is an augmented Bergman class.

For each nonempty subset $S\subseteq \E$, let $\mathrm{H}_S$ be the polymatroid on $\E$ with cage $\textbf a$ whose dual polymatroid has the simplex $\Delta_S^0$ as its independence polytope.  By \Cref{prop:dual}, the multisymmetric lift $\lift(\mathrm{H}_S)$ is the matroid on $\EE$ whose unique circuit is $\pi^{-1}(S)$.
In \cite[Section 7.2]{EHL}, it is shown that the augmented Bergman class of this matroid is equal to $h_{\pi^{-1}(S)} \in A^1(X_\EE)$.  We thus conclude that $[\Sigma_{\mathrm{H}_S}] = h_S$ by \Cref{lem:pullback} and \Cref{lem:augpullback}.
\end{proof}

\begin{remark}\label{rem:schubert}
Arguing as in \cite[Section 7.2]{EHL}, one can show that the set of monomials
\[
\{h_{F_1}^{d_1}\cdots h_{F_k}^{d_k}: \emptyset\subsetneq F_1 \subsetneq \cdots \subsetneq F_k \subseteq \E,\ d_1 \leq |\pi^{-1}(F_1)| \text{ and } d_i < |\pi^{-1}(F_i\setminus F_{i-1})| \text{ for all $2\leq i\leq k$}\}
\]
form a $\ZZ$-basis for $A^\bullet(X_{\textbf a})$.
Moreover, combining with \Cref{thm:polymatroidintersect}, one can further show that these monomials are equal to the augmented Bergman classes of polymatroids whose multisymmetric lifts are $\Gamma_{\textbf a}$-fixed \emph{Schubert matroids} on ground set $\EE$. In particular, $A^{\bullet}(X_{\textbf{a}})$ is generated by the augmented Bergman classes of realizable polymatroids with cage $\textbf{a}$. This basis can also be obtained from the techniques of \cite{DF10} and Theorem~\ref{thm:val}.
\end{remark}

\subsection{Augmented Chow rings of polymatroids}
This subsection records the properties of the augmented Chow ring of a polymatroid, but is not logically necessary for subsequent sections of this paper.
{The non-augmented version of the following theorem appeared in \cite{PP23, CHLSW}}

\begin{theorem}\label{thm:Kahler}
Let $\ell \in A^1(X_{\Sigma_{\P}})$ be an element corresponding to a strictly convex piecewise linear function on $\Sigma_{\P}$. Then the following hold:
\begin{enumerate}
   \item (Poincar\'{e} duality) There is an isomorphism $\deg_{\P} \colon A^{r}(X_{\Sigma_{\P}}) \to \mathbb{Z}$ such that, for $0 \le k \le r/2$, the pairing 
   $$A^k(X_{\Sigma_{\P}}) \times A^{r - k}(X_{\Sigma_{\P}}) \to \mathbb{Z}, \quad (x, y) \mapsto \deg_{\P}(xy)$$
   is unimodular. 
   \item (Hard Lefschetz) For every $0 \le k \le r/2$, the map
   $$A^k(X_{\Sigma_{\P}}) \otimes \mathbb{Q} \to A^{r - k}(X_{\Sigma_{\P}}) \otimes \mathbb{Q}, \quad x \mapsto \ell^{r - 2k} x$$
   is an isomorphism.
   \item (Hodge-Riemann) For every $0 \le k \le r/2$, the bilinear form
   $$A^k(X_{\Sigma_{\P}}) \otimes \mathbb{Q} \times A^k(X_{\Sigma_{\P}}) \otimes \mathbb{Q} \to \mathbb{Q}, \quad (x, y) \mapsto (-1)^k \deg_{\P}(\ell^{r - 2k} xy)$$
   is positive definite on the kernel of multiplication by $\ell^{r - 2k + 1}$. 
\end{enumerate}
\end{theorem}

\begin{proof}
The support of $\Sigma_{\P}$ is the same at the support of the Bergman fan of $\lift(\P) \times 0$. The result then follows from \cite[Theorem 1.6]{ADH22} and \cite{AHK18}. For more details, see \cite[Proof of Corollary 4.7]{CHLSW}.
\end{proof}

As $X_{\Sigma_{\P}}$ is a subvariety of $X_{\textbf{a}}$, there is a restriction map $A^{\bullet}(X_{\textbf{a}}) \to A^{\bullet}(X_{\Sigma_{\P}})$.
{We often extend the degree map of Theorem~\ref{thm:Kahler} to the whole Chow ring $\deg_\P \colon A^\bullet(X_{\Sigma_\P}) \to \ZZ$ by declaring it to be zero on the lower-degree graded components.}
The degree map satisfies the following version of the projection formula: for any $x \in A^{\bullet}(X_{\textbf{a}})$, the degree of the image of $x$ in $A^{\bullet}(X_{\Sigma_{\P}})$ is equal to the degree in $A^{\bullet}(X_{\textbf{a}})$ of $x \cdot [\Sigma_{\P}]$. 

\begin{corollary}\label{cor:identify}
The kernel of $A^{\bullet}(X_{\textbf{a}}) \to A^{\bullet}(X_{\Sigma_{\P}})$ is $\operatorname{ann}([\Sigma_{\P}])$, so we may identify $A^{\bullet}(\P)$ with $A^{\bullet}(X_{\Sigma_{\P}})$. 
\end{corollary}

\begin{proof}
By Poincar\'{e} duality, an element $x \in A^{k}(X_{\textbf{a}})$ is in the kernel of the map to $A^{\bullet}(X_{\Sigma_{\P}})$ if and only if, for all $y \in A^{n - r - k}(X_{\textbf{a}})$, $\deg(x  \cdot [\Sigma_{\P}] \cdot y) = 0$. By Poincar\'{e} duality on $A^{\bullet}(X_{\textbf{a}})$, we see that $x \cdot [\Sigma_{\P}] = 0$. Therefore the kernel of $A^{\bullet}(X_{\textbf{a}}) \to A^{\bullet}(X_{\Sigma_{\P}})$ is $\operatorname{ann}([\Sigma_{\P}])$. 
\end{proof}

\begin{corollary}\label{cor:augChow}
We have that 
$$A^{\bullet}(\P)= \frac{\mathbb{Z}[x_F, y_i : F \text{ flat, }i \in \E \text{ non-loop}]}{\mathcal{I}_1 + \mathcal{I}_2 + \mathcal{I}_3 + \mathcal{I}_4}, \text{ where }$$
$$\mathcal{I}_1 = \langle x_{F_1} x_{F_2} : F_1, F_2 \text{ incomparable flats}\rangle, \quad \mathcal{I}_2 = \langle \prod_{i \in S} y_i^{a_i} : a_i > 0, \sum {a_i} > \operatorname{rk}_{\operatorname{P}}(S) \rangle,$$
$$\mathcal{I}_3 = \langle \prod_{i \in T} y_i^{a_i} x_F : T \cap F = \emptyset, a_i > 0, \operatorname{rk}_{\operatorname{P}}(F \cup T) \le \operatorname{rk}_{\operatorname{P}}(F) + \sum a_i \rangle, \text{ and } \mathcal{I}_4 = \langle y_i - \sum_{F \not \ni i} x_F \rangle. $$
\end{corollary}

\begin{proof}
As $X_{\Sigma_{\P}}$ is a toric variety, its Chow ring is generated by classes corresponding to rays of $\Sigma_{\P}$, with monomial relations coming from minimal non-faces of the simplicial complex given by the faces of $\Sigma_{\P}$ and a linear relation for each element of $\EE$. The rays of $\Sigma_{\P}$ correspond to non-loops of $\EE$ and flats of $\P$. For $j_1, j_2$ non-loops in $\EE$ with $\pi(j_1) = \pi(j_2)$, the relation $\be_{j_1} - \be_{j_2}$ implies that the corresponding divisor classes are equal. 

Every non-face of the complex of cones in $\Sigma_{\P}$ contains either $\{F_1, F_2\}$ for $F_1, F_2$ incomparable, $\{j_1, \dotsc, j_{k}\}$ with $\operatorname{rk}_{\operatorname{P}}(\pi(j_1, \dotsc, j_k)) < k$, or $\{j_1, \dotsc, j_{\ell}, F\}$ for $\pi^{-1}(F)$ disjoint from $\{j_1, \dotsc, j_{\ell}\}$ and $\operatorname{rk}_{\operatorname{P}}(F \cup \pi(\{j_1, \dotsc, j_\ell\})) \le \operatorname{rk}_{\operatorname{P}}(F) + \ell$. Putting this all together implies the result. 
\end{proof}

\subsection{Augmented wonderful varieties of polymatroids}\label{ssec:augwondvar}

We sketch the geometric origins of the notions introduced in this section.
Recall that, given a realization $L \subseteq V = \bigoplus_{i\in \E}V_i$ of a polymatroid $\P$, its augmented wonderful variety $W_L$ is the closure of $L$ in $\prod_{\emptyset\subsetneq S \subseteq \E} \PP(\bigoplus_{i\in S}V_i \oplus \kk)$.
In the proof of \Cref{prop:toricisom}, we described $X_{\textbf a}$ as a sequence of blow-ups from $\PP(V\oplus \kk)$ along centers disjoint from $V \subset \PP(V\oplus \kk)$.  Hence, we have a natural inclusion of $V$ into $X_{\textbf a}$, and the variety $W_L$ is equivalently the closure of $L\subseteq V$ in $X_{\textbf a}$.

\begin{proposition}\label{prop:realizable}
Let $L \subseteq \bigoplus_{i \in \E} V_i$ be a realization of a polymatroid $P$ with cage $\textbf{a}$. Then the homology class $[W_L]$ is equal to $[\Sigma_{\P}]$. 
\end{proposition}

\begin{proof}
Because $GL_{\textbf a} = \prod_{i\in \E} GL(V_i)$ is connected, its action on $A_{\bullet}(X_{{\textbf{a}}})$ is trivial, so for any $g \in GL_{\textbf a}$, we have that $[W_L] = [g \cdot W_L] = [W_{g \cdot L}]$. If we choose a general $g \in GL_{\textbf a}$, then since $\kk$ is infinite, $g \cdot L$ is general with respect to the (fixed) choice of isomorphisms $V_i  \overset{\sim}{\to} \kk^{\pi^{-1}(i)}$, so $g \cdot L \subseteq \kk^{\EE}$ is a realization of $\lift(\P)$. 

By \cite[Corollary 5.11(3)]{EHL}, the homology class of the closure of $g \cdot L$ in $X_{\EE}$ is $[\Sigma_{\lift(\P)}]$. As $u \colon X_{\EE} \to X_{{\textbf{a}}}$ is an isomorphism over $g \cdot L$, we have $u_*[\Sigma_{\lift(\P)}] = [W_{g \cdot L}]$. By Lemma~\ref{lem:augpullback}, $[\Sigma_{\lift(\P)}] = u^*[\Sigma_{\P}]$, so the result follows because $u_*u^*$ is the identity (\Cref{lem:pullback}).
\end{proof}

\begin{remark}\label{rmk:chowequiv}
The closure of $L$ in $X_{\Sigma_{\P}} \subset X_{\textbf{a}}$ is $W_L$, and the restriction map $A^{\bullet}(X_{\Sigma_{\P}}) \to A^{\bullet}(W_L)$ is an isomorphism. Indeed, the iterated blow-up description of $W_L$ implies that $A^{\bullet}(W_L)$ is generated as a ring by the restriction of $h_E$ and the classes of strict transforms of exceptional divisors on $W_L$, so the restriction map $A^{\bullet}(X_{\textbf{a}}) \to A^{\bullet}(W_L)$ is surjective. As $W_L$ is the union of strict transforms of exceptional divisors and $L$, the inclusion $W_L \hookrightarrow X_{\textbf{a}}$ factors through $X_{\Sigma_{\P}}$. Therefore the restriction map $A^{\bullet}(X_{\textbf{a}}) \to A^{\bullet}(W_L)$ factors through $A^{\bullet}(X_{\Sigma_{\P}})$, so $A^{\bullet}(X_{\Sigma_{\P}}) \to A^{\bullet}(W_L)$ is surjective.
By \cite[Proposition 3.5]{Gross}, $A^{\bullet}(W_L)$ satisfies Poincar\'{e} duality.
A surjective map between Poincar\'{e} duality algebras of the same dimension is an isomorphism, so we conclude by Theorem~\ref{thm:Kahler}(1). 
\end{remark}

\section{The exceptional isomorphism}

In this section, we deduce the isomorphism $\bigoplus_{r\geq 0}\operatorname{Val}_{r}(\textbf a) \simeq \bigoplus_{r\geq 0} A_r(X_{\textbf a})$ of graded abelian groups in \Cref{thm:val}.
An intermediary object is the Grothendieck ring $K(X_{\textbf{a}})$ of vector bundles on $X_{\textbf{a}}$, which admits a polyhedral description as a polytope algebra.

\subsection{The polytope algebra}\label{ssec:polytopealg}
Let us review the polytope algebra \cite{McM89} and its relationship to the $K$-ring of a smooth projective toric variety \cite{Mor93}, following  \cite[Appendix A]{EHL}.

\smallskip
For a subset $S\subseteq \RR^\ell$, recall that $\one_S\colon \RR^\ell \to \ZZ$  denotes its indicator function.
Let $\Sigma$ be a projective fan in $\RR^\ell$ that is unimodular over $\ZZ^\ell$.  It defines a projective toric variety $X_\Sigma$.
A (lattice) polytope $Q\subseteq \RR^\ell$ is said to be a \emph{(lattice) deformation} of $\Sigma$ if its normal fan $\Sigma_{Q}$ coarsens $\Sigma$.

\begin{definition}
Let $\mathbb I(\Sigma)$ be the subgroup of $\ZZ^{(\RR^\ell)}$ generated by $\{\one_{Q} \mid Q \text{ a lattice deformation of $\Sigma$}\}$, and let $\operatorname{transl}(\Sigma)$ be the subgroup of $\mathbb I(\Sigma)$ generated by $\{\one_Q - \one_{Q+u} \mid u \in \ZZ^\ell\}$.  We define the \emph{polytope algebra} to be the quotient
\[
\overline{\mathbb I}(\Sigma) = \mathbb I(\Sigma) / \operatorname{transl}(\Sigma).
\]
\end{definition}

For a lattice deformation $Q$, denote by $[Q]$ its class in the polytope algebra $\overline{\mathbb I}(\Sigma)$.  The multiplication in the polytope algebra is induced by Minkowski sum, that is, by $[Q_1]\cdot[Q_2] = [Q_1+ Q_2]$.
As mentioned in \Cref{ssec:nef}, a correspondence between lattice deformations of $\Sigma$ and nef toric divisors on $X_\Sigma$ \cite[Chapter 6]{CLS11} associates to each lattice deformation $Q$ a nef divisor $D_Q$.  This identifies the polytope algebra with the $K$-ring as follows.

\begin{theorem}\label{thm:EHLpolytopealgebra} \cite[Theorem A.10]{EHL}
There is an isomorphism $\overline{\mathbb I}(\Sigma) \overset\sim\to K(X_\Sigma)$ defined by $[Q] \mapsto [\mathcal O_{X_\Sigma}(D_Q)]$.
\end{theorem}

This isomorphism implies that a refinement of fans induces an injection of polytope algebras.

\begin{proposition}\label{prop:injectK}
Let $\Sigma$ and $\Sigma'$ be projective unimodular fans such that $\Sigma$ refines $\Sigma'$, so a lattice deformation $Q$ of $\Sigma'$ is also a lattice deformation of $\Sigma$.
Then, the map $\overline{\mathbb I}(\Sigma') \to \overline{\mathbb I}(\Sigma)$ that sends $[Q]\in \overline{\mathbb I}(\Sigma')$ to
$[Q]\in \overline{\mathbb I}(\Sigma)$ is injective.
\end{proposition}

\begin{proof}
Let $f\colon X_{\Sigma} \to X_{\Sigma'}$ be the corresponding toric birational map of the toric varieties induced by the map of fans $\Sigma \to \Sigma'$.
The given map $\overline{\mathbb I}(\Sigma') \to \overline{\mathbb I}(\Sigma)$, under the isomorphism of \Cref{thm:EHLpolytopealgebra}, is the pullback map $f^*\colon K(X_{\Sigma'}) \to K(X_\Sigma)$.  Its injectivity now follows from \cite[Theorem 9.2.5]{CLS11} and the projection formula.
\end{proof}

Applying \Cref{thm:EHLpolytopealgebra} to the polystellahedral variety $X_{\textbf a}$, noting that deformations of the polystellahedral fan $\Sigma_{\textbf a}$ are exactly expansions of polymatroids on $E$ (\Cref{prop:deformations}), we have the following.

\begin{corollary}\label{cor:polytopealgebra}
The map sending an expanded polymatroid $\expand(\P)$ on $\EE$ to $[\mathcal{O}_{X_{\textbf{a}}}(D_{{\expand(\P)}})]$ defines an isomorphism $\overline{\mathbb I}(\Sigma_{\textbf a}) \simeq K(X_{\textbf a})$.
\end{corollary}

We will thus use these two notions, the polytope algebra and the $K$-ring, interchangeably for the polystellahedral varieties.
We will use \Cref{prop:injectK} in conjunction with the following method of ``breaking up'' a $K$-class on a polystellahedral variety into smaller pieces when considered as a $K$-class on the stellahedral variety.

\begin{proposition}\label{prop:cubes}
Let $\P$ be a polymatroid on $\E$ of rank $r\leq n$, and let $\P'$ be the polymatroid with cage $\textbf a$ defined by $I(\P') = I(\P) \cap \prod_{i\in E} [0,a_i]$.
Then, the class $[I(\expand(\P))] \in \overline{\mathbb I}(\Sigma_\EE)$ is equal to a linear combination $[I(\lift(\P'))] + \sum_k a_k [I(\M_k)]$ where the $\M_k$ are matroids on $\EE$ of rank strictly less than $r$.
\end{proposition}

That $\P'$ is a polymatroid is explained above Theorem~\ref{thm:polymatroidintersect}. We will need the following lemma.

\begin{lemma}\label{lem:cubes}\cite[Lemma 7.3]{EHL}
An intersection of the independence polytope $I(\P) \subset \RR^E$ with an integral translate of the unit cube $[0,1]^E$, if nonempty, is an integral translate of $I(\M)$ for some matroid $\M$ on $E$.
\end{lemma}

\begin{proof}[Proof of \Cref{prop:cubes}]
By tiling $\RR^\EE$ by integral translates of the unit cube $[0,1]^\EE$, we obtain a polyhedral subdivision of $I(\expand(\P))$, with every cell of the subdivision being integral translates of $I(\M)$ for some matroid $\M$ on $\EE$ by \Cref{lem:cubes}.  By \Cref{lem:projection}, the polytope $I(\lift(\P'))$ is one of the maximal interior cells of this subdivision.
All other interior cells of the subdivision are of the form $I(\M) + v$ for $0\neq v \in \ZZ_{\geq 0}^\EE$, which implies that such matroids $\M$ are of rank strictly less than $r$ since $\expand(\P)$ has rank $r$.
\end{proof}

\subsection{The exceptional isomorphism}
We now use the map $u \colon X_{\EE} \to X_{\textbf{a}}$ to construct an exceptional ring isomorphism $\phi_{\textbf a} \colon K(X_{\textbf a}) \overset\sim\to A^\bullet(X_{\textbf a})$.
Its ``exceptional'' nature is that it differs from the Chern character map, which is an isomorphism $ch\colon K(X)\otimes \QQ \to A^\bullet(X)\otimes \QQ$ for any smooth projective variety $X$.
Similar exceptional isomorphisms appeared in \cite{BEST,EHL,LLPP}.  We prepare by recalling the case of $\textbf a = (1,\ldots, 1)$ established in \cite{EHL}.

\begin{theorem}\label{thm:stellaHRR} \cite[Theorem 1.8]{EHL}
There is a unique ring isomorphism $\phi_{\EE}\colon K(X_{\EE}) \to A^\bullet(X_{\EE})$ such that $\phi_{\EE}([\mathcal O_{X_\EE}(h_S)]) = 1+h_S$ for all nonempty $S\subseteq \EE$.
Moreover, for any matroid $\M$ on $\EE$ of rank $r$, the map $\phi_\EE$ satisfies
\[
\phi_{\EE}([I(\M)]) = \xi_0 + \xi_1 + \cdots + \xi_{r}
\]
where $\xi_i \in A^i(X_{\EE})$ for all $i$ and $\xi_{r} = [\Sigma_{\M^\perp}]$.
\end{theorem}

The generalization to cage $\textbf a$ is as follows.
Recall that we have a birational toric map $u\colon X_{\EE} \to X_{\textbf{a}}$ induced by the fact that the fan $\Sigma_{\textbf a}$ is a coarsening of $\Sigma_{\EE}$.

\begin{theorem}\label{thm:polystellaHRR}
There exists a (necessarily unique) isomorphism $\phi_{\textbf a}\colon K(X_{\textbf a}) \overset\sim\to A^\bullet(X_{\textbf a})$ such that we have a commuting diagram
\begin{center}
\begin{tikzcd}
K(X_{{\textbf a}}) \arrow[r, "\phi_{\textbf a}"] \arrow[d, hook, "u^*"']
& A^{\bullet}(X_{{\textbf a}}) \arrow[d, hook, "u^*"] \\
K(X_{{\EE}}) \arrow[r, "\phi_{\EE}"]
& A^{\bullet}(X_{{\EE}}).
\end{tikzcd}
\end{center}
Moreover, for any polymatroid $\P$ on $E$ with cage $\textbf a$ and rank $r$, the map $\phi_{\textbf a}$ satisfies 
\[
\phi_{\textbf a}\big([I(\expand(\P))] \big) = \xi_0 + \xi_1 + \cdots + \xi_{r}
\]
where $\xi_i \in A^i(X_{\textbf a})$ for all $i$ and $\xi_{r} = [\Sigma_{\P^\perp}]$.
\end{theorem}

\begin{proof}
That the two vertical maps are injections follows from \Cref{lem:pullback} and \Cref{prop:injectK}.
With these injections, we now need to show that the map $\phi_{\EE}$ restricts to give a well-defined map $\phi_{\textbf a}$ that is surjective.
Recall that the Chow ring $A^\bullet(X_{\textbf a})$ is generated by the simplicial generators $h_S$.
We claim that $K(X_{\textbf a})$ is also generated as a ring by the line bundles $[\mathcal O_{X_{\textbf a}}(h_S)]$.
Both the well-definedness and the surjectivity of $\phi_{\textbf a}$ would then follow from \Cref{thm:stellaHRR} since $u^*h_S = h_{\pi^{-1}(S)}$ by \Cref{lem:pullback}.

For the claim, one notes that for any deformation $Q$ of a projective unimodular fan $\Sigma$, the inverse $[Q]^{-1}$ of the class $[Q]\in \overline{\mathbb I}(\Sigma)$ is a polynomial in $[Q]$.  See for instance \cite[Proof of Lemma A.12]{EHL}.  The claim thus follows because the simplicial generators form a basis of $A^1(X_{\textbf a})$.

For the second statement about $\phi_{\textbf a}\big([I(\expand(\P))]\big)$, consider $[I(\expand(\P))]$ as an element of $K(X_{\EE})$ via the injection $u^*$.
\Cref{prop:cubes} and \Cref{thm:stellaHRR} imply that $\phi_\EE([I(\expand(\P))]) = \xi_0 + \cdots + \xi_r$ where $\xi_i \in A^i(X_\EE)$ and $\xi_r = [\Sigma_{\lift(\P)^\perp}]$.
Lastly, \Cref{lem:augpullback} and \Cref{prop:dual} imply that $[\Sigma_{\lift(\P)^\perp}] = u^*[\Sigma_{\P^\perp}]$.
\end{proof}

\begin{remark}
Let $\chi\colon K(X_{\textbf a}) \to \ZZ$ be the sheaf Euler characteristic map.  We sketch how one can show, arguing similarly to \cite[Section 8.1]{EHL}, that the isomorphism $\phi_{\textbf a}$ satisfies
\[
\chi(\xi) = \deg_{X_{\textbf a}} \Big( \phi_{\textbf a}(\xi) \cdot \prod_{i\in \E} (1 + y_i)^{a_i} \Big) \quad\text{for all $\xi\in K(X_{\textbf a})$}.
\]
By conjugating the isomorphism $\phi_{\textbf a}$ with the map that sends the $K$-class of a vector bundle to its dual and the map that is multiplication by $(-1)^k$ on $A^{k}(X_{\textbf{a}})$, one obtains an isomorphism $\zeta_{\textbf a}$ such that $\zeta_{\textbf a}([\mathcal O_{W_L}]) = [W_L]$ for any realization $L\subseteq V$ of a polymatroid with cage $\textbf a$.
Combining \Cref{prop:realizable} with \Cref{rem:schubert}, one shows that $A^\bullet(X_{\textbf a})$ is spanned as an abelian group by $\{[W_L]: L\subseteq V\}$, and hence $\zeta_{\textbf a}$ satisfies $\chi(\xi) = \deg_{X_{\textbf a}}\big( \zeta_{\textbf a}(\xi) \cdot (1 + h_E + \cdots + h_E^n)\big)$.
One then computes that the anti-canonical divisor of $X_{\textbf a}$ is $h_E + \sum_{i\in \E} a_i y_i$, and by Serre duality concludes the desired formula.
\end{remark}

\section{Proofs of main theorems}

We now use \Cref{thm:polystellaHRR} to prove \Cref{thm:val} and \Cref{thm:HR}.

\subsection{The valuative group is isomorphic to the Chow homology group}

\begin{proof}[Proof of \Cref{thm:val}]
Since $B(\P^\perp) = -B(\P) + \textbf a$ and $I(\expand(\P)) = \big(p_\pi^{-1}(B(\P)) + \RR_{\leq 0}^\EE\big) \cap \RR_{\geq 0}^\EE$, 
the assignment $\one_{B(\P)} \mapsto \one_{I(\expand(\P^\perp))}$ gives a well-defined map $\bigoplus_{r=0}^{n} \operatorname{Val}_r(\textbf a) \to \mathbb I(\Sigma_{\textbf a})$, because all the operations --- negation, translation, inverse image, Minkowski sum, and restriction --- behave well with respect to indicator functions.
Hence, we have a map of abelian groups $\bigoplus_{r=0}^{n} \operatorname{Val}_r(\textbf a) \to K(X_{\textbf a})$ defined by $\one_{B(\P)}\mapsto [I(\expand(\P^\perp))]$.
Let $\psi$ be the composition of this map with the map $\phi_{\textbf a}\colon K(X_{\textbf a}) \to A^\bullet(X_{\textbf a})$ in \Cref{thm:polystellaHRR}. Note that $\psi$ is upper-triangular with respect to the gradings on $\bigoplus_{r=0}^{n} \operatorname{Val}_r(\textbf a)$ and $A_{\bullet}(X_{\textbf{a}})$.

\Cref{cor:augspan}, stating that $A_\bullet(X_{\textbf a})$ is spanned by $\{[\Sigma_\P] : \P \text{ a polymatroid with cage $\textbf a$}\}$, implies surjectivity of $\psi$.
For injectivity, suppose we have polymatroids $\P_1, \ldots, \P_k$ with cage $\textbf a$ and integers $c_1, \ldots, c_k$ such that $\sum_{j=1}^k c_j [\Sigma_{\P_j}] = 0$.
Then by \Cref{lem:augpullback}, the validity of \Cref{thm:val} when $\textbf a = (1, \ldots, 1)$, established in \cite[Theorem 1.5]{EHL}, implies that $\sum_j c_j \one_{B(\lift(\P_j))} = 0$.
Since each $\P_j$ has cage $\textbf a$, and since the image under the projection $p_\pi$ of the unit cube $[0,1]^{\EE}$ is the box $\prod_{i\in E} [0,a_i] \subset \RR^E$, \Cref{lem:projection} implies that $p_\pi\big(B(\lift(\P_j))\big) = B(\P_j)$.
We thus conclude $\sum_j c_j \one_{B(\P_j)} = 0$, proving the injectivity of $\psi$. Therefore $\psi$ is an isomorphism, and so the map that sends $\textbf{1}_{B(\P)}$ to $[\Sigma_\P]$ is an isomorphism. 
\end{proof}

Let $\psi$ be the map as constructed in the proof above.  Noting that polymatroid duality induces an involution of $\bigoplus_{r=0}^{n} \operatorname{Val}_r(\textbf a)$, by composing $\psi$ with the inverse $\phi_{\textbf a}^{-1}$ of the isomorphism in \Cref{thm:polystellaHRR}, we conclude the following.

\begin{corollary}
The map of abelian groups $\bigoplus_{r=0}^{n} \operatorname{Val}_r(\textbf a) \to K(X_{\textbf a})$ defined by $\one_{B(\P)}\mapsto [I(\expand(\P))]$ is an isomorphism.
\end{corollary}

\subsection{The Hall--Rado formula}\label{ssec:HR}

We first note a reinterpretation of the Hall--Rado condition.

\begin{lemma}\label{lem:HR}\cite[Theorem 2]{M75}
A collection of subsets $S_1, \ldots, S_r$ of $E$ satisfies the Hall--Rado condition with respect to a polymatroid $\P = (E,\operatorname{rk})$ of rank $r$ if and only if there exists a map $f\colon [r] \to E$ with $f(i) \in S_i$ such that $\sum_{i=1}^r \be_{f(i)} \in B(\P)$.
\end{lemma}

\begin{proof}[Proof of \Cref{thm:HR}]
For a nonempty subset $S\subseteq \E$, we showed in the proof of \Cref{cor:augspan} that if $\mathrm{H}_S$ is the polymatroid whose dual polymatroid has the simplex $\Delta_S^0$ as its independence polytope, then $[\Sigma_{\mathrm{H}_S}] = h_S$.  
Applying this to \Cref{thm:polystellaHRR}, we have $\phi_{\textbf a}([I(\expand(\mathrm{H}_S^\perp))])  = 1 +h_S$.
Thus, as the degree map $\deg_{X_{\textbf a}}$ is zero on $A^{i}(X_{\textbf a})$ for $i<n$, \Cref{thm:polystellaHRR} implies that
\[
\deg_{X_{\textbf a}}\big( \phi_{\textbf a}([I(\expand(\P^\perp))][I(\expand(\mathrm{H}_{S_1}^\perp))]\cdots[I(\expand(\mathrm{H}_{S_r}^\perp))])\big) = 
\deg_{X_{\textbf a}}\big( [\Sigma_{\P}] \cdot h_{S_1} \cdots h_{S_r} \big).
\]
Let $\widetilde \P$ be the polymatroid of rank $n$ on $\E$ whose independence polytope is $I(\P^\perp) + \Delta_{S_1}^0 + \cdots + \Delta_{S_r}^0$.
Since multiplication in the polytope algebra is Minkowski sum and expansion commutes with Minkowski sum, 
we have that $[I(\expand(\widetilde \P))] $ equals the class $[I(\expand(\P^\perp))][I(\expand(\mathrm{H}_{S_1}^\perp))]\cdots[I(\expand(\mathrm{H}_{S_r}^\perp))]$ in the left-hand-side of the equation above.
By \Cref{lem:HR} and the fact that $B(\P^\perp) = - B(\P) + \textbf a$, we have that $\textbf a\in I(\widetilde \P)$ if and only if $S_1, \ldots, S_r$ satisfies the Hall--Rado condition with respect to $\P$.
The theorem now follows from the following \Cref{lem:corner}.
\end{proof}

\begin{lemma}\label{lem:corner}
For $\widetilde \P$ a polymatroid of rank $n$ on $\E$, not necessarily with cage $\textbf a$, we have that
\[
\deg_{X_{\textbf a}}(\phi_{\textbf a}([I(\expand(\widetilde \P))])) = \begin{cases}
1 & \text{if $\textbf a \in I(\widetilde \P)$}\\
0 & \text{otherwise.}
\end{cases}
\]
\end{lemma}

\begin{proof}
Let $\widetilde \P'$ be the polymatroid with cage $\textbf a$ defined by $I(\widetilde \P') = I(\widetilde \P) \cap \prod_{i\in E} [0,a_i]$.  By \Cref{prop:cubes} and the commuting diagram in \Cref{thm:polystellaHRR}, we have that
\[
\deg_{X_{\textbf a}} (\phi_{\textbf a}([I(\expand(\widetilde \P))])) = \deg_{X_{\EE}} ([\Sigma_{\lift(\widetilde \P')^\perp}]),
\]
which is zero unless $\lift(\widetilde \P')$ has rank $n$.
When $\lift(\widetilde \P')$ has rank $n$, that is, it is the boolean matroid on $\EE$, we have that $[\Sigma_{\lift(\widetilde \P')^\perp}]$ is the class of a point in $A_0(X_\EE) = A^n(X_\EE)$, and hence $\deg_{X_{\EE}} ([\Sigma_{\lift(\widetilde \P')^\perp}]) = 1$ in this case.
Now, note that $\lift(\widetilde \P')$ has rank $n$, or equivalently $(1, \ldots, 1) \in I(\lift(\widetilde \P'))$, if and only if $\textbf a\in I(\widetilde \P')$ by \Cref{lem:projection}, and by construction $\textbf a\in I(\widetilde \P')$ if and only if $\textbf a\in I(\widetilde \P)$.
\end{proof}

\begin{proof}[Proof of \Cref{cor:basisgenfct}]
Follows from \Cref{lem:HR} and \Cref{thm:val}.
\end{proof}

\begin{remark}
At least when $\P$ is realizable, Corollary~\ref{cor:basisgenfct} implies Theorem~\ref{thm:HR}, as follows. For a realization $L \subseteq \bigoplus_{i \in \E} V_i$ of $\P$, let $V_S = \bigoplus_{i\in S} V_i$ for $\emptyset \subsetneq S\subseteq E$.  Collecting the projection maps $L \hookrightarrow \bigoplus_{i\in E} V_i \to V_S$, we obtain an inclusion
$$L \hookrightarrow \bigoplus_{\emptyset \subsetneq S \subseteq \E} V_S,$$
which is a realization of a polymatroid $\P'$ with ground set $\{S : \emptyset \subsetneq S \subseteq \E\}$.
Let $\{\mathbf f_S\}$ denote the set of standard basis vectors of $\RR^{\{S : \emptyset \subsetneq S \subseteq E\}}$, to avoid confusion with $\be_S  = \sum_{i\in S} \be_S \in \RR^E$.
A collection of subsets $S_1, \dotsc, S_r$ of $E$ satisfies $\mathbf f_{S_1} + \cdots + \mathbf f_{S_r} \in B(\P')$ if and only if it satisfies the Hall--Rado condition (with respect to $\P$), so applying Corollary~\ref{cor:basisgenfct} recovers Theorem~\ref{thm:HR}.
\end{remark}

\begin{remark}
One can also prove Corollary~\ref{cor:basisgenfct} by using Theorem~\ref{thm:val} to reduce to the case of realizable polymatroids, when Corollary~\ref{cor:basisgenfct} is \cite[Proposition 7.15]{CCMM} (and can also be deduced from \cite{LiImages}). By Remark~\ref{rem:schubert}, in order to check that two valuative functions are equal, it suffices to check on realizable polymatroids. The valuativity of $[\Sigma_{\P}]$ implies that the volume polynomial of $A^{\bullet}(\P)$ is valuative, and it is clear from the definition of valuativity that the basis generating function of a polymatroid is valuative. 
\end{remark}

\section{Polypermutohedra}

Let $\pi\colon \EE \to \E$ be with cage $\textbf a$.  The polystellahedral fan $\Sigma_\pi$ has the distinguished ray $\rho_\emptyset = \RR_{\geq 0}(-\be_\EE)$.  The star of the fan $\Sigma_\pi$ at the ray $\rho_\emptyset$ is the \emph{polypermutohedral fan} $\underline\Sigma_\pi$ introduced in \cite{CHLSW} as the Bergman fan of the boolean polymatroid with cage $\textbf a$. Explicitly, the cones of $\underline{\Sigma}_{\pi}$ are in bijection with pairs $S \le \mathcal{F}$, where $\mathcal{F} = \{\emptyset \subsetneq F_1 \subsetneq \dotsb \subsetneq F_k \subsetneq F_{k+1} = \E\}$ is a flag of proper subsets of $\E$ and $S$ is a subset of $\EE$ containing no fiber of $\pi$. 
Let $\underline X_{\textbf a}$ be the associated toric variety, which we call the \emph{polypermutohedral variety} with cage $\textbf a$, with the embedding $\iota\colon \underline X_{\textbf a} \hookrightarrow X_{\textbf a}$ as the toric divisor corresponding to the ray $\rho_\emptyset$. We set $\underline{X}_{\emptyset} = \mathrm{pt}$. 

\medskip
Suppose $\P$ is a polymatroid with cage $\textbf a$ and rank $r$.
We note the following fact about the pullback $\iota^*[\Sigma_\P]\in A_{r-1}(\underline X_{\textbf a})$.
The augmented Bergman fan $\Sigma_\P$ contains the ray $\rho_\emptyset$ if and only if $\P$ is loopless.
Hence, if $\P$ has a loop, then $\iota^*[\Sigma_\P]=0$.
If $\P$ is loopless, the star of $\Sigma_\P$ at the ray $\rho_\emptyset$ is the \emph{Bergman fan} $\underline \Sigma_\P$ of $\P$ introduced in \cite[Definition 1.6]{CHLSW}.  It is an $(r-1)$-dimensional balanced subfan of $\underline \Sigma_\pi$, and the resulting the \emph{Bergman class} $[\underline \Sigma_\P] \in A_{r-1}(\underline X_{\textbf a})$ equals the pullback $\iota^*[\Sigma_\P]$.

\medskip
Using Bergman fans and Bergman classes of loopless polymatroids, we establish analogues of the main theorems \Cref{thm:val} and \Cref{thm:HR} in the polypermutohedral setting.

\subsection{The valuative group of loopless polymatroids}

Define a subgroup of $\operatorname{Val}_r(\textbf a)$ by
\[
\operatorname{Val}_r^{\circ}(\textbf a) = \text{ the subgroup generated by }\{\one_{B(\P)}: \P \text{ a loopless polymatroid with cage $\textbf a$ and rank $r$}\}.
\]
Note that $\operatorname{Val}_0^{\circ}(\textbf a) = 0$.
We have the following analogue of \Cref{thm:val}.

\begin{theorem}\label{thm:permloopless}
For any $1\leq r \leq n$, the map that sends a loopless polymatroid $\P$ with cage $\textbf a$ and rank $r$ to the Bergman class $[\underline\Sigma_\P]$ induces an isomorphism $\operatorname{Val}_r^\circ(\textbf a) \overset\sim\to A_{r-1}(\underline X_{\textbf a})$.
\end{theorem}

We will deduce \Cref{thm:permloopless} from \Cref{thm:val} by identifying the kernel of the map $\operatorname{Val}_r(\textbf a) \overset\sim\to A_r(X_{\textbf a}) \overset{\iota^*}\to A_{r-1}(\underline X_{\textbf a})$ with the subgroup of $\operatorname{Val}_{r}(\textbf a)$ generated by polymatroids with loops.
An alternate proof that does not rely on \Cref{thm:val} but proceeds by developing the polypermutohedral analogue of \Cref{thm:polystellaHRR} is sketched in \Cref{rem:altperm}.

\medskip
Before proving Theorem~\ref{thm:permloopless}, we relate the Poincar\'{e} polynomial of the polystellahedral variety to the Poincar\'{e} polynomials of polypermutohedral varieties. For $J \subseteq E$, let $\textbf{a} \setminus J$ be the vector obtained by removing the entries corresponding to $J$. Recall that $\underline{X}_{\emptyset}$ is a point. 

\begin{lemma}\label{lem:hpoly}
We have that 
$$\sum_{i=0}^{n} \operatorname{rank} A^i(X_{\textbf{a}}) t^i =  t^{n} \operatorname{rank} A^0(\underline{X}_{\emptyset}) + \sum_{\emptyset \subseteq J \subsetneq E}  t^{|\pi^{-1}(J)|} \sum_{i=0}^{n - |\pi^{-1}(J)| - 1}\operatorname{rank} A^i(\underline X_{\textbf{a} \setminus J}) t^i .$$
\end{lemma}

\begin{proof}
As the Poincar\'{e} polynomial of a smooth projective toric variety is the $h$-polynomial of its fan, it is enough to show that
$$f({\Sigma_{\textbf{a}}})(t) = (1 + t)^n + \sum_{\emptyset \subseteq J \subsetneq E} t(1 + t)^{|\pi^{-1}(J)|} f({\underline{\Sigma}_{\textbf{a} \setminus J}}) (t),$$
where $f(\Sigma)$ is the $f$-polynomial of a fan $\Sigma$. We prove this bijectively. To each cone $\sigma$ of some $\underline{\Sigma}_{\textbf{a} \setminus J}$ corresponding to a pair $S \le \mathcal{F}$, we obtain $2^{|\pi^{-1}(J)|}$ cones of $\Sigma_{\textbf{a}}$ by adding $J$ to every element of the flag and then adding all $2^{|\pi^{-1}(J)|}$ possible subsets of $\pi^{-1}(J)$ to $S$. When $J \not= E$ and we add $k$ elements to $S$, this gives a cone of dimension $\operatorname{dim} \sigma + k + 1$. When $J = E$ and we add $k$ elements to $S$, this gives a cone of dimension $k$. 
\end{proof}

\begin{proof}[Proof of Theorem~\ref{thm:permloopless}]
For any $i\in \E$, a polymatroid base polytope $B(\P)$ is always contained in the half-space $\{x\in \RR^\E : x_i \geq 0\}$, and it is contained in the hyperplane $\{x\in \RR^\E : x_i = 0\}$ if and only if $\P$ has $i$ as a loop.
Thus, the claim in the proof of \cite[Lemma 5.9]{BEST} implies that  we have a decomposition
\[
\operatorname{Val}_r(\textbf{a}) = \bigoplus_{J \subseteq E} \operatorname{Val}_r^{\circ}(\textbf{a} \setminus J)
\]
given by sending a loopless polymatroid $\P$ of rank $r$ on $E \setminus J$ to the polymatroid on $E$ with $\operatorname{rk}(S) = \operatorname{rk}_{\operatorname{P}}(S \cap J)$.
We now induct on the size of $E$, where the base case $|E|=1$ is straightforward.
Comparing the decomposition of $\operatorname{Val}_r(\textbf a)$ above with Lemma~\ref{lem:hpoly}, we see that the induction hypothesis implies $\operatorname{rank} A_{r-1}(\underline{X}_{\textbf{a}}) = \operatorname{rank} \operatorname{Val}^{\circ}_r(\textbf{a})$.

By the construction of the permutohedral fan $\underline \Sigma_\pi$ as the star of ray $\rho_\emptyset$ in $\Sigma_\pi$, every ray of $\underline\Sigma_\pi$ is the image of a ray in $\Sigma_\pi$ that forms a cone with $\rho_\emptyset$.  Hence, the pullback $\iota^*\colon A^\bullet(X_{\textbf a}) \to A^\bullet(\underline X_{\textbf a})$ is surjective because $\iota^*\colon A^1(X_{\textbf a}) \to A^1(\underline X_{\textbf a})$ is.
We thus have a surjection $\iota^*\colon A_r(X_{\textbf{a}}) \to A_{r-1}(\underline{X}_{\textbf{a}})$ that satisfies $\iota^*[\Sigma_{\P}]= [\underline{\Sigma}_{\P}]$ if $\P$ is loopless and $\iota^*[\Sigma_{\P}] = 0$ otherwise. Therefore the composition 
$$\operatorname{Val}_r^\circ(\textbf a) \to A_r(X_{\textbf{a}}) \to A_{r-1}(X_{\textbf{a}})$$
is a surjection of finite free abelian groups of the same rank, and hence is an isomorphism. 
\end{proof}

\begin{remark}\label{rem:altperm}
We sketch an alternate proof of \Cref{thm:permloopless}.
First, arguing as in \cite[Proof of Theorem D]{EFLS}, one shows an isomorphism $\bigoplus_{r = 1}^n\operatorname{Val}_r^\circ(\textbf a) \simeq K(\underline X_{\textbf a})$ when $\textbf a= (1, \ldots, 1)$, and uses it to deduce \Cref{thm:permloopless} for the $\textbf a = (1, \ldots, 1)$ case.  Now, using that polypermutohedral fans are coarsenings of the permutohedral fan $\underline \Sigma_\EE$, just as polystellahedral fans are coarsenings of the stellahedral fan, one similarly deduces the polypermutohedral analogue of \Cref{thm:polystellaHRR}.  Then, one deduces \Cref{thm:permloopless} the same way that we proved \Cref{thm:val} here.
\end{remark}

\subsection{The dragon Hall--Rado formula}\label{ssec:DHR}
Let $\underline X_{\textbf a}$ be the polypermutohedral variety, with the embedding $\iota\colon \underline X_{\textbf a} \hookrightarrow X_{\textbf a}$ as the toric divisor corresponding to the ray $\rho_{\emptyset}$. The following theorem generalizes \cite[Theorem 5.2.4]{BES}.

\begin{theorem}\label{thm:dHR}
For a polymatroid $\P = (\E,\rk_{\P})$ of rank $r$, a collection of subsets $S_1, \ldots, S_{r-1}$ is said to satisfy the \emph{dragon Hall--Rado} condition if
\[\rk\Big( \bigcup_{j\in J} S_j\Big) \geq |J|+1 \quad\text{for all nonempty $J\subseteq [r-1]$}.
\]
Then, if $\P$ is loopless, we have
\[
\deg_{X_{\textbf a}}(h_{S_1} \cdots h_{S_{r-1}} [\Sigma_\P] [\underline X_{\textbf a}]) =
\begin{cases}
1 & \text{if the dragon Hall--Rado condition is satisfied}\\
0 & \text{otherwise}.
\end{cases}
\]
\end{theorem}

\begin{proof}
Note that, in $A^{\bullet}(X_{\textbf{a}})$, we have that $x_{\emptyset} = -\sum_{\emptyset \subsetneq S \subseteq E} (-1)^{|S|} h_S$. Then, for any $S_1, \dotsc, S_{r-1}$,
\begin{equation}\label{eq:sum}
\deg_{X_{\textbf{a}}}(h_{S_1} \dotsb h_{S_{r-1}}[\Sigma_{\P}] [\underline{X}_{\textbf{a}}]) = -\sum_{\emptyset \subsetneq S \subseteq E} (-1)^{|S|} \deg_{\operatorname{P}}(h_{S_1} \dotsc h_{S_{r-1}} h_S).
\end{equation}
Suppose we have sets $S_1, \dotsc, S_{r-1}$ that satisfy the dragon Hall--Rado condition. Because $\P$ is loopless, every term in the above sum corresponds to $r$ sets that satisfy the Hall--Rado condition, and so each term is $(-1)^{|S|}$. Because the sum is over nonempty sets, this gives the result. 

Suppose that $S_1, \dotsc, S_{r-1}$ fails the dragon Hall--Rado condition. There is some nonempty subset $T$ of $E$ such that $S_1, \dotsc, S_{r-1}, T$ fails the Hall--Rado condition; we may take $T = S_i$ for some $i$. 
Let $T_1, T_2$ be nonempty subsets of $E$ such that $S_1, \dotsc, S_{r-1}, T_1$ and $S_1, \dotsc, S_{r-1}, T_2$ both fail the Hall--Rado condition. We claim that $S_1, \dotsc, S_{r-1}, T_1 \cup T_2$ fails the Hall--Rado condition. Indeed, if there is a function $f \colon [r] \to \E$ as in Lemma~\ref{lem:HR} with $f(r) \in T_1 \cup T_2$, then $f(r)$ lies in $T_1$ or $T_2$, contradicting the assumption. 

This implies that the set $\{T : \emptyset \subsetneq T \subseteq E, \text{ } S_1, \dotsc, S_{r-1}, T \text{ fails Hall--Rado}\}$ is nonempty and has a unique maximal element. Furthermore, this set is downward closed: if $S_1, \dotsc, S_{r-1}, T$ fails the Hall--Rado condition and $\emptyset \subsetneq T' \subseteq T$, then $S_1, \dotsc, S_{r-1}, T'$ fails the Hall--Rado condition. This implies that the sum in (\ref{eq:sum}) is zero.
\end{proof}

\begin{remark}
Theorem~\ref{thm:dHR} can be alternatively proved along the lines of Theorem~\ref{thm:HR}, by using the polypermutohedral analogue of Theorem~\ref{thm:polystellaHRR} and a reformation of the dragon Hall--Rado condition in terms of a matching condition as in \cite[Proposition 5.2.3]{BES}. 
\end{remark}

\bibliographystyle{alpha}
\bibliography{polyval.bib}

\end{document}